\documentclass[a4paper,12pt]{amsart}
\usepackage[utf8]{inputenc}
\usepackage[english]{babel}
\usepackage{amsmath, amssymb, latexsym, amsfonts, bbm, stackrel,tikz,hyperref}
\usepackage[matrix,arrow,curve]{xy}
\usetikzlibrary{arrows,matrix,fit,positioning}

\numberwithin{equation}{section}
\theoremstyle{plain}
\newtheorem{theorem}[equation]{Theorem}
\newtheorem{corollary}[equation]{Corollary}
\newtheorem{lemma}[equation]{Lemma}
\theoremstyle{definition}
\newtheorem{definition}[equation]{Definition}

\theoremstyle{remark}
\newtheorem{remark}[equation]{Remark}

\newcommand{\HGr}{\mathrm{HGr}}
\newcommand{\RGr}{\mathrm{RGr}}
\newcommand{\Gr}{\mathrm{Gr}}
\newcommand{\HProj}{\mathrm{HP}}
\newcommand{\HFlag}{\mathrm{HFlag}}
\newcommand{\Hplane}{\mathrm{H}}
\newcommand{\Hstruct}{\mathcal{H}}
\newcommand{\KW}{\mathrm{KW}}
\newcommand{\GW}{\mathrm{GW}}
\newcommand{\W}{\mathrm{W}}
\newcommand{\KO}{\mathrm{KO}}
\newcommand{\Sp}{\mathrm{Sp}}
\newcommand{\Orth}{\mathrm{O}}
\newcommand{\K}{\mathrm{K}}
\newcommand{\pt}{\mathrm{pt}}
\newcommand{\Tr}{\mathrm{Tr}}
\newcommand{\T}{\mathrm{T}}
\newcommand{\Z}{{\mathbb{Z}}}
\newcommand{\Q}{{\mathbb{Q}}}
\newcommand{\A}{\mathbb{A}}
\newcommand{\Sph}{\mathrm{S}}
\newcommand{\SSph}{\mathbb{S}}

\newcommand{\thc}{\operatorname{th}}
\newcommand{\triv}{\mathbbm{1}}
\newcommand{\SH}{\mathcal{SH}}

\newcommand{\Hom}{\operatorname{Hom}}
\newcommand{\id}{\operatorname{id}}
\newcommand{\Spec}{\operatorname{Spec}}
\newcommand{\Ho}{{\mathcal{H}_{\bullet}}}
\newcommand{\HP}{\mathcal{HP}}
\newcommand{\IQ}{\mathrm{IQ}}
\newcommand{\PE}{\mathrm{PE}}

\newcommand{\cprod}{\star}

\newcommand{\bigslant}[2]{{\left.\raisebox{.2em}{$#1$}\middle/\raisebox{-.2em}{$#2$}\right.}}

\begin{document}

\title[Operations and cooperations in derived Witt theory]{Stable operations and cooperations in derived Witt theory with rational coefficients}

\author{Alexey Ananyevskiy}

\address{Institute for Advanced Study, Princeton, NJ, USA}

\address{Chebyshev Laboratory\\ St. Petersburg State University\\ 14th Line, 29b, Saint Petersburg\\ 199178 Russia}

\email{alseang@gmail.com}



\begin{abstract}
The algebras of stable operations and cooperations in derived Witt theory with rational coefficients are computed and an additive description of cooperations in derived Witt theory is given. The answer is parallel to the well-known case of $\K$-theory of real vector bundles in topology. In particular, we show that stable operations in derived Witt theory with rational coefficients are given by the values on the powers of Bott element. 
\end{abstract}

\maketitle

\section{Introduction}

Derived Witt theory introduced by Balmer \cite{Bal99} (see also \cite{Bal05} for an extensive survey) immerses Witt groups of (commutative, unital) rings and, more generally, Witt groups of schemes, into the realm of generalized cohomology theories, producing for a smooth variety $X$ a sequence of groups $W^{[n]}(X)$. This sequence is $4$-periodic in $n$ with $\W^{[0]}(X)$ and $\W^{[2]}(X)$ being canonically identified with the Witt group of symmetric vector bundles and the Witt group of symplectic vector bundles respectively. The latter groups were introduced by Knebusch \cite{Kne77}. All the groups $\W^{[n]}(X)$ are presented by generators and relations: roughly speaking, one should repeat the classic definition of Witt group of a field in the setting of derived categories of coherent sheaves (or perfect complexes) carefully treating the notion of metabolic objects. The above mentioned periodicity yields that in a certain sense we do not have "higher" derived Witt groups in distinction to the case of algebraic $\K$-theory.

Another approach to derived Witt theory is given by higher Grothedieck-Witt groups $\GW^{[n]}_i(X)$ (also known under the name of hermitian $\K$-theory) defined for schemes by Schlichting \cite{Sch10b}, see also \cite{Sch10a, Sch12}. For an affine scheme these groups coincide with hermitian $\K$-groups introduced by Karoubi. It turns out \cite[Proposition~6.3]{Sch12} that negative higher Grothendieck-Witt groups coincide with the derived Witt groups defined by Balmer, $\GW^{[n]}_i(X)\cong \W^{[n-i]}(X)$ for $i<0$.

If the characteristic of the base field is not $2$, then higher Grothendieck-Witt groups of smooth varieties are representable in the stable motivic homotopy category, see \cite{Hor05} or \cite{ST13} for a geometric model. It is well-known that derived Witt theory can be obtained from higher Grothendieck-Witt groups inverting the Hopf element $\eta$, see for example \cite[Theorem~6.5]{An12}). The Hopf element $\eta$ is the element in the motivic stable homotopy group $\pi_{1,1}(k)$ corresponding to the projection $\A^2-\{0\}\to \mathbb{P}^1$, $(x,y)\mapsto [x:y]$ (see Definition~\ref{def_hopf}). Thus derived Witt theory is represented in the stable motivic homotopy category by a spectrum representing higher Grothendieck-Witt groups with $\eta$ inverted. We denote the latter spectrum $\KW$. This spectrum is not only $(1,1)$-periodic via $\eta$ but also $(8,4)$-periodic with the periodicity realized by cup product with a Bott element $\beta\in \KW^{-8,-4}(\pt)$. In this paper we compute the algebras of operations and cooperations in derived Witt theory with rational coefficients, that is $\KW_\Q^{*,*}(\KW_\Q)$ and $(\KW_\Q)_{*,*}(\KW_\Q)$, and give an additive description of the cooperations in derived Witt theory, $\KW_{*,*}(\KW)$ (see Definition~\ref{def:bigr} for the notation). The answer is the following one (see Theorems~\ref{thm:rational_stable_operations_KW},~\ref{thm:cooperations},~\ref{thm:cooperations_integral}).

\begin{theorem}
Let $k$ be a field of characteristic not $2$. Then the homomorphism of left $\KW_\Q^{0,0}(\Spec k)\cong \W_\Q(k)$-modules
\[
Ev\colon \KW_\Q^{0,0}(\KW_\Q) \to \prod_{m\in \Z} \W_\Q(k)
\]
given by 
\[
Ev(\phi)=\left(\dots,\beta^2\phi(\beta^{-2}),\beta\phi(\beta^{-1}),\phi(1),\beta^{-1}\phi(\beta),\beta^{-2}\phi(\beta^{2}),\dots \right)
\]
is an isomorphism of algebras. Here the product on the left is given by composition and the product on the right is the component-wise one. 

Moreover, $\KW_\Q^{p,q}(\KW_\Q)=0$ when $4\nmid p-q$ and the above isomorphism induces an isomorphism of left $\KW_\Q^{*,*}(\Spec k)\cong \W_\Q(k)[\eta^{\pm 1},\beta^{\pm 1}]$-modules
\[
\KW_\Q^{*,*}(\KW_\Q)\cong \bigoplus_{r,s\in\Z} \beta^r\eta^s\prod_{m\in \Z} \W_\Q(k)
\]
with $\deg \beta =(-8,-4)$, $\deg \eta =(-1,-1)$.
\end{theorem}

\begin{theorem} \label{thm:intrcoop}
Let $k$ be a field of characteristic not $2$. Then the homomorphism of $\W_\Q(k)[\eta^{\pm 1}]\cong \bigoplus\limits_{n\in\Z}\KW_\Q^{n,n}(\Spec k)$-algebras
\[
\W_\Q(k)[\eta^{\pm 1}][\beta^{\pm 1}_l,\beta_r^{\pm 1}]\to (\KW_\Q)_{*,*}(\KW_\Q)
\]
given by 
\[
\beta_l\mapsto \Sigma^{8,4}\beta\wedge u_{\KW_\Q},\quad \beta_r\mapsto u_{\KW_\Q}\wedge \Sigma^{8,4}\beta
\]
is an isomorphism of rings. Here $u_{\KW_\Q}\colon \SSph\to \KW_\Q$ is the unit map and 
\begin{multline*}
\Sigma^{8,4}\beta\wedge u_{\KW_\Q},\,u_{\KW_\Q}\wedge \Sigma^{8,4}\beta \in (\KW_\Q)_{8,4}(\KW_\Q)=\\
=\Hom_{\SH(k)}(\SSph\wedge \Sph^{8,4},\KW_\Q\wedge \KW_\Q).
\end{multline*}
\end{theorem}

\begin{theorem} 
Suppose that $k$ is a field of characteristic not $2$ and let $\mathrm{M}$ be the abelian subgroup of $\Q[v,v^{-1}]$ generated by polynomials 
\[
f_{j,n}=\frac{v^{-n}\prod_{i=0}^{j-1}(v-(2i+1)^2)}{4^j(2j)!},
\]
$j\ge 0, n\in \Z$. Then there are canonical isomorphisms of left $\KW^{0,0}(\Spec k)\cong \W(k)$-modules
\[
\KW_{p,q}(\KW) \cong \left[ \begin{array}{ll} \W(k)\otimes_\Z \mathrm{M}, & 4\mid p-q, \\ 0, & \text{otherwise}. \end{array} \right.
\]
\end{theorem}

These theorems show that the algebras of stable operations and cooperations in derived Witt theory with rational coefficients have structure similar to the well-known case of (topological) $\K$-theory of real vector bundles $\KO^{top}$. This is not an accidental coincidence; these theories have quite a lot in common. $\KO^{top}$ is built out of real vector bundles and every real vector bundle over a compact space admits a scalar product providing an isomorphism with the dual bundle. Derived Witt theory, roughly speaking, is built out of vector bundles with a fixed isomorphism with the dual bundle. In the motivic setting the element $\eta$ is invertible in derived Witt theory. Real points of the Hopf map give a double cover of $\mathrm{S}^1$, i.e. real points of $\eta$ correspond to $2\in\Z\cong\pi_0^{st}$. Thus $\KO^{top}_{1/2}$ ($\K$-theory of real vector bundles with inverted $2$) should be a nice approximation to derived Witt groups. It is well-known that $(\KO^{top}_{1/2})^{n}$ is $4$-periodic in $n$ with 
\[
(\KO^{top}_{1/2})^{0}(\pt)=\Z[\tfrac{1}{2}],\quad (\KO^{top}_{1/2})^{n}(\pt)=0,\, n=1,2,3.
\]
The same holds for derived Witt theory: $\W^{[n]}$ is $4$-periodic in $n$ with 
\[
\W^{[0]}(\pt)=\W(k),\quad \W^{[n]}(\pt)=0,\,n=1,2,3.
\]
In fact, over the real numbers one can show that the (real) realization functor takes the motivic spectrum $\KW$ to the spectrum $\KO^{top}_{1/2}$ and there are deep theorems comparing $\W^{[n]}(X)$ to $(\KO^{top})^{n}(X(\mathbb{R}))$ for an algebraic variety $X$ over the field of real numbers, see \cite{Br84,KSW15}. Moreover, in a private communication Oliver R\"ondigs outlined to me a strategy how one can obtain a description of $(\KW\otimes \Z[\tfrac{1}{2}])_{*,*}(\KW\otimes \Z[\tfrac{1}{2}])$ and $\KW_\Q^{*,*}(\KW_\Q)$ over a base field of characteristic zero applying Brumfiel's theory \cite{Br84} to the well-known computation of cooperations and rational operations in topology \cite{AHS71}.

The algebras of stable operations and cooperations in $\KO_\Q^{top}$ can be described as follows. Denote $\beta^{top}\in (\KO^{top}_\Q)^{-4}(\pt)$  the element inducing periodicity $(\KO_\Q^{top})^{n+4}\cong (\KO_\Q^{top})^n$. Every stable operation is uniquely determined by its values on the powers of $\beta^{top}$ yielding an isomorphism 
\[
(\KO_\Q^{top})^*(\KO_\Q^{top})\cong \bigoplus_{n\in \Z} (\beta^{top})^n \prod_{m\in \Z} \Q,
\]
while for the cooperations one has
\[
(\KO_\Q^{top})_*(\KO_\Q^{top})\cong \Q[\beta_l^{\pm 1}, \beta_r^{\pm 1}],
\]
where $\beta_r$ and $\beta_l$ are similar to the ones from Theorem~\ref{thm:intrcoop}.

Computations of $(\KO_\Q^{top})^*(\KO_\Q^{top})$ and $(\KO_\Q^{top})_*(\KO_\Q^{top})$ could be carried out quite easily using Serre's theorem about finiteness of stable homotopy groups of spheres. In the motivic setting the analogous result on stable homotopy groups is not completely settled, moreover, our motivation is just the opposite one. It was pointed out to me by Marc Levine that the above computations of stable operations and cooperations in $\KW_\Q$ combined with the technique developed by Cisinski and Deglise \cite{CD09} could possibly yield the motivic version of Serre's finiteness. This problem is addressed in a forthcoming paper \cite{ALP15}. 

Our approach to the computation of stable operations and cooperations in $\KW_\Q$ and cooperations in $\KW$ is straightforward. The spectrum $\KW$ is obtained by localization from the spectrum $\KO$ representing higher Grothendieck-Witt groups, hence
\begin{gather*}
\KW^{*,*}_\Q(\KW_\Q)=\KW^{*,*}_\Q(\KO),\quad (\KW_\Q)_{*,*}(\KW_\Q)=(\KW_\Q)_{*,*}(\KO),\\
\KW_{*,*}(\KW)=\KW_{*,*}(\KO).
\end{gather*}
The odd spaces in the spectrum $\KO$ are all the same and coincide with the infinite quaternionic Grassmannian $\HGr$. Derived Witt theory of $\HGr$ is known to be given by power series in characteristic classes  \cite[Theorem~9.1]{PW10b}. The pullbacks along the structure maps of $\KO$ can be described explicitly using the following computation of characteristic classes of triple tensor product of rank two symplectic bundles  (Lemma~\ref{lm:formal_group_law}).
\begin{lemma}
Let $E_1,E_2$ and $E_3$ be rank $2$ symplectic bundles over a smooth variety $X$. Put $\xi_i = b_1^{\KW}(E_i)\in \KW^{4,2}(X)$ and denote $\xi(n_1,n_2,n_3)$ the sum of all the monomials lying in the orbit of $\xi_1^{n_1}\xi_2^{n_2}\xi_3^{n_3}$ under the action of $S_3$. Then
\begin{align*}
& b^{\KW}_1(E_1\otimes E_2 \otimes E_3)=\beta \xi(1,1,1),\\
& b^{\KW}_2(E_1\otimes E_2 \otimes E_3)=\beta\xi(2,2,0)-2\xi(2,0,0),\\
& b^{\KW}_3(E_1\otimes E_2 \otimes E_3)=\beta\xi(3,1,1) -8\xi(1,1,1),\\
& b^{\KW}_4(E_1\otimes E_2 \otimes E_3)=\beta\xi(2,2,2)+ \xi(4,0,0)- 2\xi(2,2,0).
\end{align*}
\end{lemma}
\noindent This computation is a derived Witt analogue of the equality 
\[
c_1^\K(L_1\otimes L_2)=c_1^\K(L_1)+c_1^\K(L_2)-c_1^\K(L_1)c_1^\K(L_2)
\]
in $\K$-theory, i.e. an analogue of a formal group law. It turns out that the inverse limit $\varprojlim \KW^{*+8n+4,*+4n+2}_\Q(\HGr)$ can be easily computed yielding the desired answer, while the $\varprojlim^1$ term vanishes. For the cooperations we use the same strategy with the result of Panin and Walter on the derived Witt theory of $\HGr$ replaced by Theorem~\ref{thm:homological_HGr} providing us with the following description of derived Witt homology of $\HGr$ (see Definitions~\ref{def:HGr_Hspace},~\ref{def:gamma} for the details).
\begin{theorem}
Let $k$ be a field of characteristic not $2$. Then there is a canonical isomorphism of $\KW^{*,*}(\Spec k)\cong \W(k)[\eta^{\pm 1},\beta^{\pm 1}]$-algebras
\[
\KW_{*,*}(\HGr_+)\cong \W(k)[\eta^{\pm 1},\beta^{\pm 1}][x_1,x_2,\dots].
\]
\end{theorem}

The paper is organized in the following way. In Section~2 we recall the well-known definitions and constructions in generalized (co)homology theories representable in the stable motivic homotopy category introduced by Morel and Voevodsky. The next section deals with the definitions and basic properties of cup and cap product in motivic setting. In Section~4 we recall the theory of symplectic orientation in generalized motivic cohomology developed by Panin and Walter \cite{PW10a,PW10b}. The fifth section is dual to the forth one and deals with symplectically oriented homology theories. In Sections~6 and~7 we recall various representability properties of higher Grothendieck-Witt groups and derived Witt theory. In Section~8 we compute characteristic classes of triple tensor product of rank $2$ symplectic bundles. In the last two sections we compute the algebras of stable operations and cooperations in $\KW_\Q$ and give an additive description of cooperations in $\KW$.

{\it Acknowledgement.} I would like to express my sincere gratitude to I.~Panin and M.~Levine for the conversations on the subject of the paper and to O.~R\"ondigs for encouraging me to deal with the integral cooperations. I would although like to thank the anonymous referee for valuable comments and suggestions. The final part of the research was done during my stay in the Institute for Advanced Study and was supported by the National Science foundation under agreement No. 1128155. The research is partially supported by RFBR grants 15-01-03034 and 16-01-00750, by PJSC ``Gazprom Neft'' and by ``Dynasty'' foundation.

\section{Recollection on generalized motivic (co)homology}
In this section we recall some basic definitions and constructions in the unstable and stable motivic homotopy categories $\Ho (k)$ and $\SH (k)$. We refer the reader to the foundational papers \cite{MV99,V98} for an introduction to the subject. We use the version of stable motivic homotopy category based on $\HP^1$-spectra introduced in \cite{PW10c}.

Throughout this paper $k$ is a field of characteristic different from $2$. 

\begin{definition}
Let $\mathrm{Sm}/k$ be the category of smooth varieties over $k$. A \textit{motivic space} over $k$ is a simplicial presheaf on $\mathrm{Sm}/k$. Each $X\in \mathrm{Sm}/k$ defines an unpointed motivic space $\Hom_{\mathrm{Sm}/k}(-,X)$ constant in the simplicial direction. We will often write $\pt$ for $\Spec k$ regarded as a motivic space. Inverting all the weak motivic equivalences in the category of the pointed motivic spaces we obtain the \textit{pointed motivic unstable homotopy category} $\Ho (k)$.
\end{definition}

\begin{definition}
Denote $\Sph^{1,1}=(\A^1-\{0\},1)$, $\Sph^{1,0}=\Sph^1_s=\Delta^1/\partial(\Delta^1)$ and 
\[
\Sph^{p+q,q}=(\Sph^{1,1})^{\wedge q}\wedge (\Sph^{1,0})^{\wedge p}
\]
for the integers $p,q\ge 0$. Denote $\T=\A^1/(\A^1-\{0\})$  the Morel-Voevodsky object which is canonically isomorphic to $\Sph^{2,1}$ in $\Ho(k)$ \cite[Lemma~3.2.15]{MV99}.
\end{definition}

\begin{definition}
Let 
$V=(k^{\oplus 4}, \phi)$, $\phi(\mathbf{x},\mathbf{y})=x_1y_2-x_2y_1+x_3y_4-x_4y_3$, be the standard symplectic vector space over $k$ of dimension $4$. The \textit{quaternionic projective line} $\HProj^1$ is the variety of symplectic planes in $V$. Alternatively, it can be described as $\HProj^1=\bigslant{\Sp_4}{\Sp_2\times \Sp_2}$. Denote $*=\langle e_1,e_2\rangle \in \HProj^1(k)$ for the standard basis $e_1,e_2,e_3,e_4$ of $V$. If not otherwise specified we consider $\HProj^1$ as a pointed motivic space $(\HProj^1,*)$. Let $\HP^1$ be the pushout of
\[
\A^1 \xleftarrow{0} \pt \xrightarrow{*} \HProj^1.
\]
There is an obvious isomorphism $\HP^1\xrightarrow{\simeq} \HProj^1$ in $\Ho(k)$ given by a contraction of $\A^1$ and we usually identify these two objects in $\Ho(k)$.
\end{definition}

\begin{remark}
The only reason that we need $\HP^1$ is Definition~\ref{def:KO} since the morphisms that we use there do no exist for $\HProj^1$.
\end{remark}

\begin{lemma}[{\cite[Theorem~9.8]{PW10c}}]
There exists a canonical isomorphism $\HProj^1\cong \T\wedge \T$ in $\Ho(k)$ .
\end{lemma}
\begin{corollary}\label{cor:HP1_sphere}
There exists a canonical isomorphism $\HProj^1\cong \Sph^{4,2}$ in $\Ho(k)$.
\end{corollary}
\begin{proof}
Follows from the lemma applying the canonical isomorphism $\T\cong \Sph^{2,1}$ \cite[Lemma~3.2.15]{MV99}.
\end{proof}

\begin{definition}
An \textit{$\HP^1$-spectrum $A$} is a sequence of pointed motivic spaces $(A_0,A_1,A_2,\dots)$ equipped with structural maps $\sigma_n\colon \HP^1\wedge A_n\to A_{n+1}$. A morphism of $\HP^1$-spectra is a sequence of morphisms of pointed motivic spaces compatible with the structural maps. Inverting the stable motivic weak equivalences as in \cite{Jar00} we obtain the motivic stable homotopy category $\SH(k)=\SH_{\HP^1}(k)$. This category has a canonical symmetric monoidal structure. From now on by a spectrum we mean an $\HP^1$-spectrum.
\end{definition}

\begin{lemma}[{\cite[Theorem~12.1]{PW10c}}]
The stable homotopy categories of $\T$-spectra and of $\HP^1$-spectra are equivalent.
\end{lemma}

\begin{definition}
Every pointed motivic space $Y$ gives rise to a suspension spectrum
\[
\Sigma^\infty_{\HP^1} Y=(Y,\HP^1\wedge Y,(\HP^1)^{\wedge 2}\wedge Y,\dots).
\] 
Put $\mathbb{S}=\Sigma^\infty_{\HP^1} \pt_+$  for the sphere spectrum.
\end{definition}

\begin{definition}
Let $A=(A_0,A_1,\dots)$ be an $\HP^1$-spectrum and $m$ be an integer. Denote $A\{m\}=(A\{m\}_0,A\{m\}_1,\dots)$ the spectrum given by 
\[
A\{m\}_n=
\left[ 
\begin{array}{ll} 
A_{m+n}, &  m+n\ge 0, \\ 
\pt & m+n<0
\end{array} 
\right.
\]
with the structure maps induced by the structure maps of $A$.
\end{definition}

\begin{definition}
It follows from Corollary~\ref{cor:HP1_sphere} that in $\SH(k)$ there is a canonical isomorphism $(A\wedge \Sph^{4,2})\{-1\}\cong A$. The suspension functors $-\wedge \Sph^{p+q,q}, p,q\ge 0,$ become invertible in $\SH(k)$, so we extend the notation to arbitrary integers $p,q$ in an obvious way.
\end{definition}

\begin{definition} \label{def:bigr}
For $A,B\in \SH(k)$ put 
\begin{gather*}
A^{i,j}(B)=\Hom_{\SH(k)}(B,A\wedge \Sph^{i,j}),\quad A^{*,*}(B)=\bigoplus_{i,j \in \Z} A^{i,j}(B),\\
A_{i,j}(B)=\Hom_{\SH(k)}(\mathbb{S}\wedge \Sph^{i,j},A\wedge B),\quad A_{*,*}(B)=\bigoplus_{i,j\in \Z} A_{i,j}(B).
\end{gather*}
Let $f\colon B\to B'$ be a morphism in $\SH(k)$. Denote 
\[
f^A\colon A^{*,*}(B')\to A^{*,*}(B),\quad f_A\colon A_{*,*}(B)\to A_{*,*}(B')
\]
the natural morphisms given by composition with $f$.
\end{definition}

\begin{remark}
Using suspension spectra we may treat every pointed motivic space as a spectrum, in particular, we may treat a smooth variety $X$ as a suspension spectrum $\Sigma^\infty_{\HP^1} (X_+,+)$. Thus all the definitions involving $A^{*,*}(B)$ and $A_{*,*}(B)$ are applicable to the case of $B$ being a pointed motivic space or a smooth variety.
\end{remark}

\begin{definition}
For $A,B\in \SH(k)$ we have \textit{suspension isomorphisms} 
\[
\Sigma^{p,q}\colon A^{*,*}(B)\xrightarrow{\simeq} A^{*+p,*+q}(B\wedge \Sph^{p,q}),\, \Sigma^{p,q}\colon A_{*,*}(B)\xrightarrow{\simeq} A_{*+p,*+q}(B\wedge \Sph^{p,q})
\]
given by smash-product $-\wedge \id_{\Sph^{p,q}}$. The isomorphisms from \cite[Lemma~3.2.15]{MV99} and Corollary~\ref{cor:HP1_sphere} induce \textit{suspension isomorphisms}
\begin{gather*}
\Sigma_\T\colon A^{*,*}(B)\xrightarrow{\simeq} A^{*+2,*+1}(B\wedge \T),\, \Sigma_\T\colon A_{*,*}(B)\xrightarrow{\simeq} A_{*+2,*+1}(B\wedge \T),\\
\resizebox{\textwidth}{!}{$
\Sigma_{\HProj^1}\colon A^{*,*}(B)\xrightarrow{\simeq} A^{*+4,*+2}(B\wedge \HProj^1),\, \Sigma_{\HProj^1}\colon A_{*,*}(B)\xrightarrow{\simeq} A_{*+4,*+2}(B\wedge \HProj^1),
$}\\
\resizebox{\textwidth}{!}{$
\Sigma_{\HP^1}\colon A^{*,*}(B)\xrightarrow{\simeq} A^{*+4,*+2}(B\wedge \HP^1),\, \Sigma_{\HP^1}\colon A_{*,*}(B)\xrightarrow{\simeq} A_{*+4,*+2}(B\wedge \HP^1).
$}
\end{gather*}
We write $\Sigma^n_\T, \Sigma^n_{\HProj^1}$ and $\Sigma^n_{\HP^1}$ for the $n$-fold composition of the respective suspension isomorphisms.
\end{definition}

\begin{definition}
Let $A=(A_0,A_1,\dots)$ be an $\HP^1$-spectrum. Denote $\Tr_n A$ the spectrum given by 
\[
(\Tr_n A)_m =
\left[ 
\begin{array}{ll} 
A_m, &  m\le n, \\ 
(\HP^1)^{\wedge m-n} \wedge A_m, & m> n
\end{array} 
\right.
\]
with the structure maps induced by the structure maps of $A$.
\end{definition}

\begin{remark}
The obvious map $\Sigma^{\infty}_{\HP^1} A_n\{-n\}\to \Tr_n A$ clearly becomes an isomorphism in $\SH (k)$.
\end{remark}

\begin{lemma} \label{lm:truncated_limit}
Consider $A\in \SH(k)$ and let $B=(B_0,B_1,\dots)$ be an $\HP^1$-spectrum with structure maps $\sigma_n\colon \HP^1\wedge B_n\to B_{n+1}$. Then
\begin{enumerate}
\item
the canonical homomorphism
\[
\varinjlim A_{*+4n,*+2n}(B_n)\to A_{*,*}(B)
\]
is an isomorphism, here the limit is taken with respect to the morphisms
\[
(\sigma_n)_A\circ\Sigma_{\HP^1}\colon A_{*+4n,*+2n}(B_n)\to A_{*+4(n+1),*+2(n+1)}(B_{n+1});
\]
\item
there is a short exact sequence 
\[
0\to \varprojlim\nolimits^1 A^{*+4n-1,*+2n}(B_n) \to A^{*,*}(B) \to \varprojlim_{} A^{*+4n,*+2n}(B_n) \to 0
\]
where the limit is taken with respect to the morphisms
\[
\Sigma_{\HP^1}^{-1}\circ \sigma_n^A\colon A^{*+4(n+1),*+2(n+1)}(B_{n+1})\to A^{*+4n,*+2n}(B_{n}).
\]
\end{enumerate}
\end{lemma}
\begin{proof}
Straightforward, using $B= \varinjlim \Tr_n B$ and a mapping telescope. In the motivic setting see, for example, \cite[Lemma~A.34]{PPR09}.
\end{proof}

\section{Cup and cap product on generalized motivic (co)homology}

In this section we recall the well-known constructions of cup and cap product in generalized (co)homology. A classic reference for this theme in (non-motivic) stable homotopy theory is \cite[III.9]{Ad74}.

\begin{definition}
A \textit{commutative ring spectrum} $A$ is a commutative monoid $(A,m_A\colon A\wedge A\to A,u_A\colon\SSph\to A)$ in $(\SH(k),\wedge,\mathbb{S})$.
\end{definition}

\begin{definition}
Let $(A,m_A,u_A)$ be a commutative ring spectrum and $f\colon B\to C\wedge D$ be a morphism in $\SH(k)$. The \textit{cup product}
\[
\cup_f \colon A^{p,q}(C)\times A^{i,j}(D) \to A^{p+i,q+j}(B)
\]
is given by $a\cup_f b= (m_A\wedge \sigma)\circ(\id_A\wedge \tau_{\Sph^{p,q},A}\wedge \id_{\Sph^{i,j}} ) \circ (a \wedge b) \circ f$,
\[
a\cup_f b =
\left(
\vcenter{\xymatrix{
 B \ar[r]^(0.45)f & C \wedge D   \ar[r]^(0.35){a\wedge b} & A\wedge \Sph^{p,q}\wedge A\wedge \Sph^{i,j} \ar[lld]_(0.7)*!/u3pt/{\labelstyle \id_A \wedge \tau_{\Sph^{p,q},A}\wedge \id_{\Sph^{i,j}} } \\ 
A\wedge A \wedge \Sph^{p,q} \wedge \Sph^{i,j} \ar[rr]^(0.6){m_A\wedge \sigma } & & A\wedge \Sph^{p+i,q+j}
}}
\right), 
\]
where $ \tau_{\Sph^{p,q},A}\colon \Sph^{p,q}\wedge A\xrightarrow{\simeq} A\wedge \Sph^{p,q}$ and $\sigma\colon \Sph^{p,q}\wedge \Sph^{i,j}\xrightarrow{\simeq} \Sph^{p+i,q+j}$ are permutation isomorphisms.  We will usually omit the subscript $f$ from the notation when the morphism is clear from the context. The cup product is clearly bilinear and associative. We are going to use this product in the following special cases:
\begin{enumerate}
\item
Let $U_1,U_2\subset X$ be open subsets of a smooth variety $X$ and 
\[
f\colon X/(U_1\cup U_2)\to X/U_1 \wedge X/U_2
\]
be the morphism induced by the diagonal embedding. Then the above construction gives a cup product
\[
\cup\colon A^{p,q}(X/U_1)\times A^{i,j}(X/U_2) \to A^{p+i,q+j}(X/(U_1\cup U_2)).
\]
In particular, for $U_1=U_2=\emptyset$ we obtain a product
\[
\cup\colon A^{p,q}(X)\times A^{i,j}(X) \to A^{p+i,q+j}(X)
\]
endowing $A^{*,*}(X)$ with a ring structure.
\item
Consider $B\in \SH(k)$ and let 
\[
f_1=\id\colon B\to B\wedge \mathbb{S},\quad f_2=\id\colon B\to \mathbb{S}\wedge B
\]
be the identity maps. Then we obtain cup products
\begin{gather*}
\cup\colon A^{p,q}(B)\times A^{i,j}(\pt) \to A^{p+i,q+j}(B),\\
\cup\colon A^{p,q}(\pt)\times A^{i,j}(B) \to A^{p+i,q+j}(B)
\end{gather*}
endowing $A^{*,*}(B)$ with a structure of $A^{*,*}(\pt)$-bimodule.
\end{enumerate}
\end{definition}

\begin{definition}
Let $\tau_{\Sph^{2,1},\Sph^{2,1}} \colon \Sph^{2,1}\wedge \Sph^{2,1} \xrightarrow{\simeq} \Sph^{2,1}\wedge \Sph^{2,1}$ be the permutation isomorphism and let $(A,m_A,u_A)$ be a commutative ring spectrum. Put
\[
\varepsilon = \Sigma^{-4,-2} \tau_{\Sph^{2,1},\Sph^{2,1}}^A \Sigma^{4,2} u_A\in A^{0,0}(\pt).
\]
Note that $\varepsilon^2=1$.
\end{definition}

\begin{lemma}\label{lm:epsilon_commutativity}
Let $(A,m_A,u_A)$ be a commutative ring spectrum and let $f\colon B\to C\wedge D$ be a morphism in $\SH(k)$. Denote $f^\tau =\tau \circ f\colon B\to D\wedge C$ with $\tau \colon C\wedge D \xrightarrow{\simeq} D\wedge C$ being the permutation isomorphism. Then
\[
a\cup_f b=(-1)^{pi}\varepsilon^{qj}  b\cup_{f^\tau} a \in  A^{p+i,q+j}(B)
\]
for every $a\in A^{p,q}(C)$ and $b\in A^{i,j}(D)$.
\end{lemma}
\begin{proof}
Examining the definition one notices that 
\[
a\cup_f b=(\Sigma^{-p-i,-q-j}\tau_{\Sph^{p,q}, \Sph^{i,j}}) \cup b\cup_{f^\sigma} a,
\]
where $\tau_{ \Sph^{p,q}, \Sph^{i,j}}\colon \Sph^{p,q}\wedge \Sph^{i,j}\xrightarrow{\simeq} \Sph^{i,j}\wedge \Sph^{p,q}$ is the permutation isomorphism. By classical homotopy theory one has $\Sigma^{-2,0}(\tau_{\Sph^{1,0},\Sph^{1,0}})=-1$, so the claim follows.
\end{proof}

\begin{definition} \label{def:cap}
Let $(A,m_A,u_A)$ be a commutative ring spectrum and $f\colon B\to C\wedge D$ be a morphism in $\SH(k)$. The \textit{cap product}
\[
\cap_f \colon A^{p,q}(C)\times A_{i,j}(B) \to A_{i-p,j-q}(D)
\]
is given by $a\cap_f x= \Sigma^{-p,-q}((m_A\wedge \tau_{\Sph^{p,q},D})\circ(\id_A\wedge a \wedge \id_D) \circ (\id_A \wedge f) \circ x)$,
\[
a\cap_f x=\Sigma^{-p,-q}\left(
\vcenter{\xymatrix{
\mathbb{S}\wedge  \Sph^{i,j} \ar[r]^x & A\wedge B \ar[r]^(0.45){\id_A\wedge f} & A\wedge C\wedge D \ar[lld]_(0.7)*!/u3pt/{\labelstyle \id_A \wedge a \wedge \id_D} \\ A\wedge A\wedge \Sph^{p,q} \wedge D \ar[rr]^(0.6){m_A\wedge \tau_{\Sph^{p,q},D}} & & A\wedge D \wedge \Sph^{p,q}
}}
\right), 
\]
where $ \tau_{\Sph^{p,q},D}\colon \Sph^{p,q}\wedge D\xrightarrow{\simeq}  D\wedge \Sph^{p,q}$ is the permutation isomorphism. The subscript $f$ will be usually omitted from the notation when the morphism is clear from the context. We are going to use this product in the following special cases:
\begin{enumerate}
\item
Let $Y$ be a pointed motivic space and let $f=\Delta \colon Y\to Y\wedge Y$ be the diagonal embedding. Then we obtain the cap product
\[
\cap\colon A^{p,q}(Y)\times A_{i,j}(Y) \to A_{i-p,j-q}(Y).
\]
One can easily check that $(a \cup a') \cap x= a \cap (a' \cap x)$.  This product endows $A_{*,*}(Y)$ with a left $A^{*,*}(Y)$-module structure.
\item
Let $U$ be an open subset of a smooth variety $X$ and $f\colon X/U\to (X/U) \wedge X_+$ be the morphism induced by the diagonal embedding. Then we obtain the cap product
\[
\cap\colon A^{p,q}(X/U)\times A_{i,j}(X/U) \to A_{i-p,j-q}(X).
\]
\item
Consider $B\in \SH(k)$ and let $f=\id \colon B\to B\wedge \mathbb{S}$ be the identity morphism. Then we obtain the \textit{Kronecker pairing}
\[
\langle -,- \rangle \colon A^{p,q}(B)\times A_{i,j}(B) \to A_{i-p,j-q}(\pt)\cong A^{p-i,q-j}(\pt).
\]
\item
Consider $B\in \SH(k)$ and let $f=\id \colon B\to \mathbb{S} \wedge B$ be the identity morphism. Then we obtain a cap product
\[
\cap \colon A^{p,q}(\pt)\times A_{i,j}(B)\to  A_{i-p,j-q}(B).
\]
endowing $A_{*,*}(B)$ with a left $A^{*,*}(\pt)$-module structure.
\end{enumerate}
\end{definition}

\begin{lemma}\label{lm:cap_naturality}
Let $A$ be a commutative ring spectrum. Then for a commutative square in $\SH(k)$
\[
\xymatrix{
C\wedge D \ar[r]^{r\wedge s} & C'\wedge D' \\
B \ar[u]^f \ar[r]^{t} & B' \ar[u]^{f'}
}
\]
and $a\in A^{*,*}(C'),\, x\in A_{*,*}(B)$  we have 
\[
s_A(r^A(a)\cap x)=a\cap t_A(x).
\]
\end{lemma}
\begin{proof}
Straightforward.
\end{proof}

\begin{definition} \label{def:duality}
Let $A$ be a commutative ring spectrum and let $p\colon X\to Y$ be a morphism of pointed motivic spaces. Then the pairing
\[
p_A \circ (-\cap-)\colon A^{*,*}(X)\times A_{*,*}(X)\to A_{*,*}(Y)
\]
is $A^{*,*}(Y)$-bilinear. Denote 
\[
D_p\colon A_{*,*}(X)\to \Hom_{A^{*,*}(Y)}(A^{*,*}(X), A_{*,*}(Y))
\]
the adjoint homomorphism of left $A^{*,*}(X)$-modules.
\end{definition}

\begin{definition} \label{def:right_cap}
Let $(A,m_A,u_A)$ and $(B,m_B,u_B)$ be commutative ring spectra. The product
\[
-\cprod-\colon A_{p,q}(B)\times A_{i,j}(B) \to A_{i+p,j+q}(B)
\]
is given by $x\cprod y = (m_A\wedge m_B) \circ (\id_A\wedge \tau_{B,A} \wedge \id_B) \circ (y\wedge x) \circ \sigma $,
\[
x\cprod y=\left(
\vcenter{\xymatrix{
\mathbb{S}\wedge \Sph^{i+p,j+q} \ar[r]^(0.45)\sigma & \mathbb{S}\wedge \Sph^{i,j} \wedge \mathbb{S}\wedge \Sph^{p,q} \ar[r]^(0.5){y\wedge x} &  A\wedge B \wedge A \wedge B  \ar[lld]_(0.7)*!/u3pt/{\labelstyle \id_A\wedge \tau_{B,A} \wedge \id_B} \\ 
 A\wedge A\wedge B \wedge B \ar[rr]^(0.6){m_A\wedge m_B} & & A\wedge B
}}
\right),
\]
where $\sigma\colon \Sph^{i+p,j+q} \xrightarrow{\simeq} \Sph^{i,j} \wedge \Sph^{p,q}$ and $\tau_{B,A}\colon B\wedge A\xrightarrow{\simeq} A\wedge B$ are the permutation isomorphisms. This product endows $A_{*,*}(B)$ with a ring structure. Moreover, one immediately checks that it agrees with the cap product introduced in the end of Definition~\ref{def:cap} under the homomorphism 
\[
A^{-p,-q}(\pt) \simeq A_{p,q}(\pt)\xrightarrow{(u_B)_A} A_{p,q}(B).
\]
\end{definition}

\section{Symplectically oriented cohomology theories}

In this section we provide a list of results from the theory of symplectic orientation in generalized motivic cohomology developed in \cite{PW10a}. 

\begin{definition}\label{def:symplectic_grassmannians}
We adopt the following notation dealing with Grassmannians and flags of symplectic spaces (cf. \cite{PW10a}).
\begin{itemize}
\item
$\Hplane_{-}=\left( k^{\oplus 2}, \begin{pmatrix} 0 & 1 \\ -1 & 0 \end{pmatrix} \right)$ is the \textit{standard symplectic plane}.
\item
$\HGr(2r,2n)=\bigslant{\Sp_{2n}}{\Sp_{2r}\times \Sp_{2n-2r}}$ is the \textit{quaternionic Grassmannian}. Alternatively, it can be described as the open subscheme of $\Gr(2r,\Hplane_{-}^{\oplus n})$ parametrizing subspaces on which the standard symplectic form is nondegenerate. 
\item
$\mathcal{U}^s_{2r,2n}$ is the tautological rank $2r$ symplectic vector bundle over $\HGr(2r,2n)$.
\item
$\HProj^n=\HGr(2,2n+2)$ is the \textit{quaternionic projective space}. 
\item
$\Hstruct(1)=\mathcal{U}^s_{2,n+2}$ is the tautological rank two symplectic vector bundle over $\HProj^n$.
\item
$\HFlag(2^{r},2n)=\bigslant{\Sp_{2n}}{\Sp_{2}\times \dots \times \Sp_{2}\times \Sp_{2n-2r}}$ is the \textit{quaternionic flag variety}. Alternatively, it can be described as the variety of flags $V_2\le V_4\le\dots V_{2r}\le \Hplane_{-}^{\oplus n}$ such that $\dim V_{2i}=2i$ and the restriction of the symplectic form is nondegenerate on $V_{2i}$ for every $i$. 
\item
$\HGr(2r,E), \HProj(E), \HFlag(2^r,E)$ are the relative versions of the above varieties defined for a rank $2n$ symplectic bundle $E$ over a smooth variety $X$.
\end{itemize}
\end{definition}

\begin{definition}[{cf. \cite[Definition~14.2]{PW10a} and \cite[Definition 12.1]{PW10b}}] \label{def:symplectic_orientation}
A \textit{symplectic orientation} of a commutative ring spectrum $A$ is a rule which assigns to each rank $2n$ symplectic bundle $E$ over a smooth variety $X$ an element $\thc(E)=\thc^A(E)\in A^{4n,2n}(E/(E-X))$ with the following properties:
\begin{enumerate}
\item
For an isomorphism $u\colon E\xrightarrow{\simeq} E'$ one has $\thc(E)=u^A \thc(E')$.
\item
For a morphism of varieties $f\colon X\to Y$, symplectic bundle $E$ over $Y$ and pull back morphism $f_E\colon f^*E\to E$ one has $f_E^A \thc(E)=\thc(f^*E)$.
\item
The homomorphisms $-\cup \thc(E)\colon A^{*,*}(X)\to A^{*+4n,*+2n}(E/(E-X))$ are isomorphisms.
\item
We have $\thc(E\oplus E')=q_1^A \thc(E)\cup q_2^A \thc(E')$ where $q_1,q_2$ are the projections from $E\oplus E'$ to its factors.
\end{enumerate}
We refer to the classes $\thc(E)$ as \textit{Thom classes}. A commutative ring spectrum $A$ with a chosen symplectic orientation is called a \textit{symplectically oriented spectrum}. 
\end{definition}

\begin{lemma}\label{lm:thom_trivial}
Let $A$ be a symplectically oriented spectrum, $X$ be a smooth variety and let $p\colon X\to \pt$ be the projection. Identify $\Hplane_{-}^{\oplus r}/(\Hplane_{-}^{\oplus r} - \{0\})\cong \T^{\wedge 2r}$. Then 
\[
\thc(p^*\Hplane_{-}^{\oplus r})=a\Sigma_\T^{2r} 1_X
\]
for some invertible element $a\in A^{0,0}(\pt)$.
\end{lemma}
\begin{proof}
We have the following isomorphisms:
\[
A^{0,0}(\pt) \xrightarrow{-\cup \Sigma_\T^{2r} 1_{\pt}} A^{4r,2r} (\T^{\wedge 2r})\cong A^{4r,2r}(\Hplane_{-}^{\oplus r}/(\Hplane_{-}^{\oplus r} - \{0\})) \xleftarrow{-\cup  \thc(\Hplane_{-}^{\oplus r})} A^{0,0}(\pt),
\]
thus $\thc(\Hplane_{-}^{\oplus r})=a \Sigma_\T^{2r} 1_{\pt}$ for some invertible $a\in A^{0,0}(\pt)$. The claim follows from the functoriality of Thom classes.
\end{proof}

\begin{remark}
There is a canonical bijection between the sets of symplectic orientations satisfying additional condition of normalization ($\thc(\Hplane_{-})=\Sigma^2_\T 1$) and homomorphisms of monoids $\mathrm{MSp}\to A$. See \cite[Theorem 12.2, 13.2]{PW10b} for the details.
\end{remark}

\begin{remark}
The main example of a symplectically oriented cohomology theory that we are interested in is higher Grothendieck-Witt groups (hermitian $\K$-theory). See Definition~\ref{def:KO} and Theorems~\ref{thm:stable_representability},~\ref{thm:KO_orientation} for the details.
\end{remark}

\begin{definition}[{\cite[Definition 14.1]{PW10a}}]
A theory of \textit{Borel classes} on a commutative ring spectrum $A$ is a rule which assigns to every symplectic bundle $E$ over a smooth variety $X$ a sequence of elements $b_i(E)=b^A_i(E)\in A^{4i,2i}(X), i\ge 1$, satisfying
\begin{enumerate}
\item
For $E\cong E'$ we have $b_i(E)=b_i(E')$ for all $i$.
\item
For a morphism of varieties $f\colon X\to Y$ and symplectic bundle $E$ over $Y$ we have $f^Ab_i(E)=b_i(f^*E)$ for all $i$.
\item
For every variety $X$ the homomorphism 
\[
A^{*,*}(X)\oplus A^{*-4,*-2}(X) \to A^{*,*}(\HProj^1\times X)
\]
given by $a+a'\mapsto p^A(a)+p^A(a')\cup b_1(\Hstruct(1))$ is an isomorphism. Here $p\colon \HProj^1\times X\to X$ is the canonical projection.
\item
$b_1(\Hplane_{-})=0\in A^{4,2}(\pt)$.
\item
For $E$ of rank $2r$ we have $b_i(E)=0$ for $i>r$.
\item
For $E,E'$ symplectic bundles over $X$ we have $b_t(E)b_t(E')=b_t(E\oplus E')$, where 
\[
b_t(E)=1+b_1(E)t+b_2(E)t^2+\dots\in A^{*,*}(X)[t].
\]
\end{enumerate}
We refer to $b_i(E)$ as \textit{Borel classes} of $E$ and $b_t(E)$ is the \textit{total Borel class}.
\end{definition}

\begin{remark}
In \cite{PW10a} the above classes were called \textit{Pontryagin classes}, but as I learned from I.~Panin it was noted by V.~Buchstaber that these classes act much more like symplectic Borel classes than Pontryagin classes in topology, so we prefer to adopt this new notation. See also \cite[Definition 7]{An15}.
\end{remark}

\begin{theorem}[{\cite[Theorem 14.4]{PW10a}}] \label{thm:bijection_orientations}
Let $A$ be a commutative ring spectrum. Then there is a canonical bijection between the set of symplectic orientations of $A$ and the set of Borel classes theories on $A$.
\end{theorem}
\begin{proof}
We give a sketch of the definition of a Borel classes theory on a symplectically oriented spectrum. First one defines $b_1(E)=z^A \thc(E)$ for a rank $2$ symplectic bundle $E$ over a smooth variety $X$ and morphism $z\colon X\to E/(E-X)$ induced by the zero section. Then the higher Borel classes are introduced using Theorem~\ref{thm:proj_bundle_cohomology} below. In particular, we have $b_r(E)=z^A \thc(E)$ for a rank $2r$ symplectic bundle $E$. See \cite{PW10a} for the details, but note that we omit the minus sign in front of $b_1(E)$.
\end{proof}

\begin{theorem}[{\cite[Theorem 8.2]{PW10a}}]\label{thm:proj_bundle_cohomology}
Let $A$ be a symplectically oriented spectrum and let $E$ be a rank $2r$ symplectic bundle over a smooth variety $X$. Denote $\HProj(E)$ the relative quaternionic projective space associated to $E$ and put $\xi=b_1(\Hstruct(1))\in A^{4,2}(\HProj(E))$. Then the homomorphism of left $A^{*,*}(X)$-modules 
\[
\bigoplus_{i=0}^{r-1} A^{*-4i,*-2i}(X) \to A^{*,*}(\HProj(E))
\]
given by $\sum\limits_{i=0}^{r-1} a_i\mapsto \sum\limits_{i=0}^{r-1} a_i\cup \xi^i$ is an isomorphism.
\end{theorem}

\begin{corollary}\label{cor:cohomological_fullflag}
Let $A$ be a symplectically oriented spectrum and let $E$ be a rank $2r$ symplectic bundle over a smooth variety $X$. Denote $\mathcal{U}_1,\mathcal{U}_2,\dots,\mathcal{U}_s$ the tautological rank $2$ symplectic bundles over $\HFlag(2^s,E)$ and put $\xi_i=b_1(\mathcal{U}_i)$. Then the homomorphism of left $A^{*,*}(X)$-modules 
\[
\bigoplus\limits_{\substack{0\le n_i\le (r-i) \\ i=1\dots s}} A^{*-4(n_1+\dots+n_s),*-2(n_1+\dots+n_s)}(X) \to A^{*,*}(\HFlag(2^s,E))
\]
given by 
\[
\sum\limits_{\substack{0\le n_i\le (r-i) \\ i=1\dots s}} a_{n_1n_2\dots n_s}\mapsto \sum\limits_{\substack{0\le n_i\le (r-i) \\ i=1\dots s}} a_{n_1n_2\dots n_s}\cup\xi_1^{n_1}\xi_2^{n_2}\dots \xi_s^{n_s}
\]
is an isomorphism.
\end{corollary}
\begin{proof}
Follows from the theorem, since one can present $\HFlag(2^s,E)$ as an iterated quaternionic projective bundle
\[
\HFlag(2^s,E) \to \HFlag(2^{s-1},E)\to \dots \to \HFlag(2,E)=\HProj(E).
\]
\end{proof}

\begin{theorem}[{\cite[Theorem 10.2]{PW10a}}] \label{thm:splitting}
Let $A$ be a symplectically oriented spectrum and let $E$ be a rank $2r$ symplectic bundle over a smooth variety $X$. Then there exists a canonical morphism of smooth varieties $f\colon Y\to X$ such that
\begin{enumerate}
\item
$f^{A}\colon A^{*,*}(X)\to A^{*,*}(Y)$ is injective,
\item
$f^*E\cong E_1\oplus E_2\oplus \hdots \oplus E_r$ for some canonically defined rank $2$ symplectic bundles $E_i$. In particular, 
\[
b_i(E)=\sigma_i(b_1(E_1),b_1(E_2),\hdots, b_1(E_r))
\]
for the elementary symmetric polynomials $\sigma_i$.
\end{enumerate}
\end{theorem}

\begin{definition}
Let $E$ be a rank $2r$ symplectic bundle over a smooth variety $X$. In the notation of Theorem~\ref{thm:splitting} we refer to $\{b_1(E_1),b_1(E_2),\dots,b_1(E_r)\}$ as \textit{Borel roots} of $E$ and denote $\xi_i=\xi_i(E)=b_1(E_i)$. Put $s_n(E)$ for the power sums of Borel roots of $E$,
\[
s_n(E) = \xi_1^n+\xi_2^n + \hdots + \xi_r^n \in A^{4n,2n}(X)
\]
and denote
\[
s_t(E)=s_1(E)t +s_2(E)t^2+\hdots \in A^{*,*}(X)[[t]].
\]
It follows from the standard relations between power sums and elementary symmetric polynomials that
\[
s_{t} (E)= -t \frac{d}{dt}\ln b_{-t}(E).
\]
\end{definition}

\begin{theorem}[{\cite[Theorem 11.2]{PW10a}}] \label{thm:cohomological_HGr_finite}
Let $A$ be a symplectically oriented spectrum. Then the homomorphism of $A^{*,*}(\pt)$-algebras
\[
\bigslant{A^{*,*}(\pt)[b_1,b_2,\dots,b_r]}{\left( h_{n-r+1},\dots, h_n \right)}\to A^{*,*}(\HGr(2r,2n))
\] 
induced by $b_i\mapsto b_i(\mathcal{U}^s_{2r,2n})$ is an isomorphism. Here $h_j=h_j(b_1,b_2,\dots,b_r)$ is the polynomial representing the $j$-th complete symmetric polynomial in $r$ variables via elementary symmetric polynomials.
\end{theorem}

\begin{definition}
We have the following ind-objects considered as pointed motivic spaces.
\begin{itemize}
\item
$\HGr(2r)=\varinjlim\limits_n (\HGr(2r,2n),*)$,
\item
$\HGr=\varinjlim\limits_{r,n} (\HGr(2r,2n), *)$,
\end{itemize}
where $*= \HGr(2,2)\in \HGr(2r,2n)$.
\end{definition}

\begin{definition}
We have the following classes over the infinite Grassmannians.
\begin{itemize}
\item
$b_i(\mathcal{U}^s_{2r})\in A^{4i,2i}(\HGr(2r))$ satisfying $b_i(\mathcal{U}^s_{2r}) |_{\HGr(2r,2n)}= b_i(\mathcal{U}^s_{2r,2n})$,
\item
$b_i(\tau^s)\in A^{4i,2i}(\HGr)$ satisfying $b_i(\tau^s)|_{\HGr(2r,2n)}= b_i(\mathcal{U}^s_{2r,2n})$.
\end{itemize}
The next theorem yields that these elements are uniquely defined by the given restrictions.
\end{definition}

\begin{definition}
Let $R$ be a graded ring and let $b_i$ be variables of degree $d_i\in \mathbb{N}$. We denote $R[[b_1,b_2,\dots]]_h$ the \textit{ring of homogeneous power series}.
\end{definition}

\begin{theorem}[{\cite[Theorem 9.1]{PW10b}}] \label{thm:cohomological_HGr}
Let $A$ be a symplectically oriented spectrum. Then the following homomorphisms of $A^{*,*}(\pt)$-algebras are isomorphisms.
\begin{enumerate}
\item
$A^{*,*}(\pt)[[b_1,b_2,\dots,b_r]]_h\xrightarrow{} A^{*,*}(\HGr(2r)_+)$, induced by $b_i\mapsto b_i(\mathcal{U}^s_{2r})$. 
\item
$A^{*,*}(\pt)[[b_1,b_2,\dots]]_h \xrightarrow{} A^{*,*}(\HGr_+)$, induced by $b_i\mapsto b_i(\tau^s)$.
\end{enumerate}
\end{theorem}

\section{Symplectically oriented homology theories}

The results of this section are dual to the results of the previous one: we compute symplectically oriented homology of quaternionic Grassmannians and flag varieties. Throughout this section $A$ denotes a symplectically oriented commutative ring spectrum in the sense of Definition~\ref{def:symplectic_orientation}.

\begin{lemma}
Let $E$ be a rank $2r$ symplectic bundle over a smooth variety $X$. Then 
\[
\thc(E)\cap - \colon A_{*,*}(E/(E-X)) \to A_{*-4r,*-2r}(X)
\]
is an isomorphism.
\end{lemma}
\begin{proof}
Using a standard Mayer-Vietoris argument we may assume that $E$ is a trivial bundle, i.e. $E=p^*\Hplane_{-}^{\oplus r}$ for the projection $p\colon X\to \pt$. By Lemma~\ref{lm:thom_trivial} $\thc(E)=a\Sigma^{2r}_\T 1_X$, thus $\thc(E)\cap -$ coincides up to an invertible scalar with the suspension isomorphism $\Sigma_\T^{-2r}$.
\end{proof}

\begin{definition}
Let $i\colon Y\to X$ be a codimension $2r$ closed embedding of smooth varieties. Suppose that the normal bundle $N_i$ is equipped with a symplectic form. The \textit{transfer map in homology} $i^{!A}$ is given by composition
\[
\resizebox{\textwidth}{!}{$
i^{!A}\colon A_{*,*}(X)\xrightarrow{p_A} A_{*,*}(X/(X-Y)) \xrightarrow{d_A} A_{*,*}(N_i/(N_i-Y)) \xrightarrow{\thc(N_i)\cap -} A_{*-4r,*-2r}(Y)
$}
\]
Here 
\begin{itemize}
\item
$X\xrightarrow{p} X/(X-Y)$ is the canonical quotient morphism,
\item
$d\colon X/(X-Y)\xrightarrow{\simeq} N_i/(N_i-Y)$ is the deformation to the normal bundle isomorphism (\cite[Theorem 3.2.23]{MV99}). 
\end{itemize}
With this notation the localization sequence in homology could be rewritten in the following way:
\[
\hdots\xrightarrow{\partial} A_{*,*}(X-Y) \xrightarrow{j_A} A_{*,*}(X) \xrightarrow{i^{!A}} A_{*-4r,*-2r}(Y) \xrightarrow{\partial}\hdots
\]
\end{definition}

\begin{lemma}\label{lm:transfer_homomorphism}
Let $i\colon Y\to X$ be a codimension $2r$ closed embedding of smooth varieties. Suppose that the normal bundle $N_i$ is equipped with a symplectic form. Then the transfer map $i^{!A}$ is a homomorphism of $A^{*,*}(X)$-modules, i.e.
\[
i^{!A}(a\cap x)=i^A(a)\cap i^{!A}(x)
\]
for every $x\in A_{*,*}(X)$ and $a\in A^{*,*}(X)$.
\end{lemma}
\begin{proof}
The morphisms $p_A$ and $d_A$ are homomorphisms of $A^{*,*}(X)$-modules by Lemma~\ref{lm:cap_naturality}, while cap product with the Thom class induces a homomorphism of $A^{*,*}(X)$-modules by Lemma~\ref{lm:epsilon_commutativity}.
\end{proof}

\begin{lemma}[{cf. \cite[Proposition 7.6]{PW10a}}]\label{lm:pullpush}
Let $E$ be a rank $2r$ symplectic bundle over a smooth variety $X$ and let $s\colon X\to E$ be a section meeting the zero section $z\colon X\to E$ transversally in $Y$. Let $i\colon Y\to X$ be the closed embedding. Equip the normal bundle $N_i$ with a symplectic form using the canonical isomorphism $i^*E\cong N_i$. Then for every $x\in A_{*,*}(X)$ we have
\[
i_{A}i^{!A}(x)= b_r(E)\cap x.
\]
\end{lemma}
\begin{proof}
Consider the following diagram.
\[
\xymatrix{
 & & A_{*,*}(N_i/(N_i-Y)) \ar[d]^{\thc(N_i)\cap -} \ar[ddl]^{j_A} \\
A_{*,*}(X) \ar[r]^(0.37){p_A} \ar@<-0.5pc>[d]_{s_A} \ar[d]^{z_A} & A_{*,*}(X/(X-Y)) \ar[d]^{s_A} \ar[ur]^{d_A} &  A_{*-4r,*-2r}(Y)  \ar[d]^{i_A} \\
A_{*,*}(E) \ar@/^-1.5pc/[u]_{\pi_A}  \ar[r]^(0.37){q_A} & A_{*,*}(E/(E-z(X))) \ar[r]^(0.55){\thc(E)\cap -} & A_{*-4r,*-2r}(X) 
}
\]
Here 
\begin{itemize}
\item
the morphisms $p_A$ and $q_A$ are induced by the quotient maps, 
\item
$d_A$ is induced by the deformation to the normal bundle isomorphism,
\item
$\pi_A$ is induced by the canonical projection $\pi \colon E\to X$,
\item
$j_A$ is induced by the isomorphism $i^*E\cong N_i$.
\end{itemize}
In the left side of the diagram $s_A$ and $z_A$ are homomorphisms inverse to an isomorphism $\pi_A$, thus $s_A=z_A$ and the left square commutes. The middle triangle commutes by the functoriality of the deformation to the normal bundle isomorphism. The right side commutes by the functoriality of Thom classes. Hence
\[
i_{A}i^{!A}(x)=i_A(\thc(N_i)\cap d_Ap_A(x))=\thc(E)\cap (q_Az_A(x)).
\]
By Lemma~\ref{lm:cap_naturality} we have
\[
\thc(E)\cap (q_Az_A(x))= z^Aq^A(\thc(E))\cap x =b_r(E)\cap x.
\]
\end{proof}

\begin{theorem}\label{thm:homological_projective_bundle}
Let $E$ be a symplectic bundle of rank $2r+2$ over a smooth variety $X$. Denote $p \colon \HProj(E)\to X$ the canonical projection and set $\xi=b_1(\Hstruct(1))$. Then the homomorphism of left $A^{*,*}(X)$-modules
\[
A_{*,*}(\HProj(E))\to \bigoplus_{n=0}^r A_{*-4n,*-2n}(X)
\]
given by $x \mapsto p_{A}(x)+p_{A}(\xi\cap x)+\dots +p_{A}(\xi^{r}\cap x)$ is an isomorphism.
\end{theorem}
\begin{proof}
A usual Mayer-Vietoris argument yields that it is sufficient to treat the case of a trivial symplectic bundle $E$, i.e. $\HProj(E)=\HProj^{r}\times X$. The proof does not depend on the base $X$, so we omit it from the notation.

By \cite[Theorems~3.1,~3.2,~3.4]{PW10a} there is a closed subvariety $Y\subset \HProj^r$ satisfying
\begin{itemize}
\item
$Y$ is a transversal intersection of a section $s\colon \HProj^r\to \Hstruct(1)$ and the zero section $z\colon \HProj^r\to \Hstruct(1)$,
\item
$\HProj^r-Y$ is $\A^1$-homotopy equivalent to a point,
\item
there is a morphism $\pi\colon Y\to \HProj^{r-1}$ which is an $\A^2$-bundle such that $\pi^*\Hstruct(1)\cong i^*\Hstruct(1)$, where $i\colon Y\to \HProj^r$ is the closed embedding.
\end{itemize}
Equip the normal bundle $N_i$ with the symplectic form induced by the isomorphism $ i^*\Hstruct(1)\cong N_i$. Identifying $A_{*,*}(\HProj^r-Y)\cong A_{*,*}(\pt)$ and $A_{*,*}(Y)\cong A_{*,*}(\HProj^{r-1})$ we obtain a long exact sequence in homology
\[
\dots\xrightarrow{\partial} A_{*,*}(\pt)\xrightarrow{j_{A}} A_{*,*}(\HProj^r) \xrightarrow{i^{!A}} A_{*-4,*-2}(\HProj^{r-1}) \xrightarrow{\partial} \dots
\]
Here $j$ is the composition $\pt\cong \HProj^r-Y \to \HProj^r$. The projection $\HProj^r\to \pt$ splits the first morphism, thus $i^{!A}$ is surjective. Denote $q\colon \HProj^{r-1}\to \pt$ the canonical projection and consider the following diagram.
\[
\xymatrix{
A_{*,*}(\pt) \ar[r]^{j_A} \ar[d]^= & A_{*,*}(\HProj^r) \ar[r]^{i^{!A}} \ar[d]^{\sum\limits_{n=0}^r p_A(\xi^n\cap -)} & A_{*-4,*-2}(\HProj^{r-1}) \ar[d]^{\sum\limits_{n=0}^{r-1} q_A(\xi^n\cap -)} \\
A_{*,*}(\pt) \ar[r]^(0.35){u} & \bigoplus\limits_{n=0}^{r} A_{*-4n,*-2n}(\pt) \ar[r]^v & \bigoplus\limits_{n=1}^{r-1} A_{*-4n,*-2n}(\pt)
}
\]
Here $u$ is the injection on the zeroth summand and $v$ is the projection forgetting about the zeroth summand. The left square commutes by Lemma~\ref{lm:cap_naturality}:
\[
\xi^n\cap j_A(x)=j_A(j^A(\xi^n)\cap x)=
\left[
\begin{array}{ll}
j_A(x), & n=0,\\
j_A(0\cap a)=0, & n>0.
\end{array}
\right.
\]
The right square commutes by Lemmas~\ref{lm:cap_naturality} and~\ref{lm:pullpush}:
\[
q_A(\xi^n\cap i^{!A}x)=p_Ai_A(\xi^n\cap i^{!A}x)=p_A(\xi^n\cap i_Ai^{!A}x)=p_A(\xi^{n+1}\cap x).
\]
The claim follows by induction.
\end{proof}

\begin{corollary}\label{cor:homological_fullflag}
Let $E$ be a symplectic bundle of rank $2r$ over a smooth variety $X$. Denote $\mathcal{U}_1,\mathcal{U}_2,\dots,\mathcal{U}_s$ the tautological rank $2$ symplectic bundles over $\HFlag(2^s,E)$. Put $\xi_i=b_1(\mathcal{U}_i)$ and denote $p\colon \HFlag(2^s,E)\to X$ the canonical projection. Then the homomorphism of $A^{*,*}(X)$-modules 
\[
A_{*,*}(\HFlag(2^{s},E))\to \bigoplus_{\substack{0\le n_i\le (r-i) \\ i=1\dots s}} A_{*-4(n_1+n_2+\dots+n_s),*-2(n_1+n_2+\dots+n_s)}(X)
\]
given by
\[
x\mapsto \sum_{\substack{0\le n_i\le (r-i) \\ i=1\dots s}} p_A((\xi_1^{n_1}\xi_2^{n_2}\dots \xi_s^{n_s})\cap x)
\]
is an isomorphism.
\end{corollary}
\begin{proof}
Follows from Theorem~\ref{thm:homological_projective_bundle}, since one can present $\HFlag(2^s,E)$ as an iterated quaternionic projective bundle
\[
\HFlag(2^s,E) \to \HFlag(2^{s-1},E)\to \dots \to \HFlag(2,E)=\HP(E)
\]
\end{proof}

\begin{theorem}\label{thm:homological_grassmannian}
Let $E$ be a symplectic bundle of rank $2r$ over a smooth variety $X$. Denote $p\colon \HFlag(2^s,E)\to X$ and $q\colon \HGr(2s,E)\to X$ the canonical projections. Then the following duality homomorphisms, given by Definition~\ref{def:duality}, are isomorphisms:
\begin{gather*}
D_p\colon A_{*,*}(\HFlag(2^s,E))\xrightarrow{} \Hom_{A^{*,*}(X)} (A^{*,*}(\HFlag(2^s,E)),A_{*,*}(X)),\\
D_q\colon A_{*,*}(\HGr(2s,E))\xrightarrow{} \Hom_{A^{*,*}(X)} (A^{*,*}(\HGr(2s,E)),A_{*,*}(X)).
\end{gather*}
\end{theorem}
\begin{proof}
The first morphism is an isomorphism by Corollaries~\ref{cor:cohomological_fullflag},~\ref{cor:homological_fullflag}.

Denote $p'\colon \HFlag(2^s,E)\to \HGr(2s,E)$ the canonical projection and shorten the notation $\mathrm{HF}=\HFlag(2^s,E),\, \mathrm{HG}=\HGr(2s,E)$. Recall that $\mathrm{HF}$ is a quaternionic flag bundle over $\mathrm{HG}$, thus
\[
D_{p'}\colon A_{*,*}(\mathrm{HF})\xrightarrow{} \Hom_{A^{*,*}(\mathrm{HG})} (A^{*,*}(\mathrm{HF}),A_{*,*}(\mathrm{HG}))
\]
is an isomorphism by the above. Since $A^{*,*}(\mathrm{HF})$ is a free $A^{*,*}(\mathrm{HG})$-module by Corollary~\ref{cor:cohomological_fullflag} it is sufficient to check that the composition
\[
\xymatrix @C=3.2pc {
A_{*,*}(\mathrm{HF}) \ar[dr]_(0.4){(D_{q})_*\circ D_{p'}} \ar[r]^(0.4){D_{p'}}_(0.4){\simeq} & \Hom_{A^{*,*}(\mathrm{HG})} (A^{*,*}(\mathrm{HF}),A_{*,*}(\mathrm{HG})) \ar[d]^{(D_{q})_*} \\
&   \Hom_{A^{*,*}(\mathrm{HG})} (A^{*,*}(\mathrm{HF}), \Hom_{A^{*,*}(X)} (A^{*,*}(\mathrm{HG}),A_{*,*}(X)))
}
\]
is an isomorphism. The claim follows from the commutativity of the following diagram, which is straightforward.
\[
\xymatrix{
 & \Hom_{A^{*,*}(\mathrm{HG})} (A^{*,*}(\mathrm{HF}), \Hom_{A^{*,*}(X)} (A^{*,*}(\mathrm{HG}),A_{*,*}(X))) \\
A_{*,*}(\mathrm{HF}) \ar[ur]^(0.35){(D_{q})_*\circ D_{p'}} \ar[dr]_(0.45){D_p}^(0.55){\simeq} &  \Hom_{A^{*,*}(X)} (A^{*,*}(\mathrm{HF})\otimes_{A^{*,*}(\mathrm{HG})} A^{*,*}(\mathrm{HG}),A_{*,*}(X))   \ar[u]_{\cong} \\
& \Hom_{A^{*,*}(X)} (A^{*,*}(\mathrm{HF}) ,A_{*,*}(X)) \ar[u]_\cong
}
\]
\end{proof}

\begin{definition} \label{def:HGr_Hspace}
The operation of orthogonal sum of symplectic bundles yields a morphism $\HGr_+\wedge \HGr_+\to \HGr_+$ endowing $A_{*,*}(\HGr_+)$ with a ring structure
\[
A_{*,*}(\HGr_+)\times A_{*,*}(\HGr_+) \to A_{*,*}(\HGr_+).
\]
\end{definition}

\begin{definition} \label{def:gamma}
For $n\ge 0$ denote $\chi_n\in A_{4n,2n}(\HProj^\infty_+)$ the unique collection of elements satisfying
\[
\langle \xi^{m},  \chi_n\rangle=
\left[
\begin{array}{ll}
1, & m=n,\\
0, & m\neq n,
\end{array}
\right.
\]
for $\xi=b_1(\Hstruct(1))$. The existence and uniqueness of these elements is guaranteed  by Theorem~\ref{thm:homological_grassmannian} (consider $s=1$). Moreover, by the same theorem we know that
$A_{*,*}(\HProj^\infty_+)$ is a free $A^{*,*}(\pt)$-module with a basis given by $\{1,\chi_1,\chi_2,\hdots\}$.
Abusing the notation we denote by the same letters the elements $\chi_n=i_A(\chi_n)\in A_{4n,2n}(\HGr_+)$ for the canonical embedding $i\colon \HProj^\infty_+ \to \HGr_+$.
\end{definition}

\begin{theorem} \label{thm:homological_HGr}
Identify 
\[
A^{*,*}(\HGr_+)\cong A^{*,*}(\pt)[[b_1,b_2,\dots]]_h \cong A^{*,*}(\pt)[[\xi_1,\xi_2,\dots]]_h^{S_\infty}
\]
by Theorems~\ref{thm:splitting} and ~\ref{thm:cohomological_HGr} via $b_i(\tau^s)\leftrightarrow b_i \leftrightarrow \sigma_i(\xi_1,\xi_2,\hdots)$. Given a partition $\lambda=\{\lambda_1\ge\lambda_2\ge\dots\ge\lambda_k>0\}$ denote $\xi(\lambda) \in A^{*,*}(\pt)[[\xi_1,\xi_2,\dots]]_h^{S_\infty}$ the sum of all the elements in the orbit of $\xi_1^{\lambda_1}\xi_2^{\lambda_2}\dots \xi_k^{\lambda_k}$. Then
\begin{enumerate}
\item
$
\langle \xi(\lambda), \chi_1^{l_1}\chi_2^{l_2}\dots\chi_r^{l_r} \rangle=
\left[
\begin{array}{ll}
1, &  l_j=\#\{\lambda_i=j\}\, \text{for all $j\ge 1$},\\
0, &  \text{otherwise},
\end{array}
\right.
$
\item
the homomorphism of $A^{*,*}(\pt)$-algebras
\[
A^{*,*}(\pt)[x_1,x_2,\dots]\to A_{*,*}(\HGr_+)
\]
induced by $x_i\mapsto \chi_i$ is an isomorphism.
\end{enumerate}
\end{theorem}
\begin{proof}
Put $|l|=l_1+l_2+\dots+l_r$ and consider the canonical embedding 
\[
i\colon (\underbrace{\HProj^\infty\times \HProj^\infty\times \hdots  \times \HProj^\infty}_{|l|})_+\to \HGr_+
\]
given by orthogonal sum. Identify
\begin{gather*}
A^{*,*}((\HProj^\infty\times \HProj^\infty\times \hdots  \times \HProj^\infty)_+)=\bigoplus_{i_j\ge 0} A^{*,*}(\pt) \xi^{i_1} \otimes \xi^{i_2} \otimes\hdots\otimes \xi^{i_l},\\
A_{*,*}((\HProj^\infty\times \HProj^\infty\times \hdots  \times \HProj^\infty)_+)=\bigoplus_{i_j\ge 0} A^{*,*}(\pt) \chi_{i_1} \otimes \chi_{i_2} \otimes\hdots\otimes \chi_{i_l}.
\end{gather*}
Put 
\[
\chi^l=\chi_1^{l_1}\chi_2^{l_2}\hdots\chi_r^{l_r},\quad \chi_{\otimes}^l=\underbrace{\chi_{1} \otimes \hdots \otimes \chi_1}_{l_1}  \otimes \underbrace{\chi_{2} \otimes \hdots \otimes \chi_2}_{l_2}  \otimes \hdots \otimes \underbrace{ \chi_{r} \otimes \hdots \otimes \chi_r}_{l_r},
\]
and denote $\xi_\otimes(\lambda)$ the sum of all the elements in the orbit of $\xi^{\lambda_1}\otimes \xi^{\lambda_2}\otimes\hdots\otimes \xi^{\lambda_l}$ under the action of $S_l$. Here $\lambda_j=0$ for $j>k$.

We have $i_A(\chi_\otimes^l)=\chi^l$ and
\[
i^A(\xi(\lambda))=
\left[
\begin{array}{ll}
0, & k>|l|,\\
\xi_\otimes(\lambda), & k\le |l|.
\end{array}
\right.
\]
By Lemma~\ref{lm:cap_naturality} we have
$\langle \xi(\lambda), \chi^l \rangle= \langle i^A(\xi(\lambda)), \chi_\otimes^l \rangle.
$

If $k>|l|$ then $i^A(\xi(\lambda))=0$ and $\langle \xi(\lambda), \chi^l \rangle=0$ by the above. 

If $k\le |l|$ then we have
\begin{multline*}
\langle \xi(\lambda), \chi^l \rangle = \langle \xi_\otimes(\lambda), \chi_\otimes^l \rangle=\\
=\sum_{\substack{(\lambda_1',\hdots,\lambda_l')=\\=(\lambda_{\sigma(1)},\hdots,\lambda_{\sigma(l)}) \\ \text{for some $\sigma\in S_l$}}}
\langle \xi^{\lambda_1'},\chi_1 \rangle \hdots \langle \xi^{\lambda_{l_1}'},\chi_1 \rangle \langle \xi^{\lambda_{l_1+1}'},\chi_2 \rangle  \dots \langle \xi^{\lambda_{l_1+l_2}'},\chi_2 \rangle\dots \langle \xi^{\lambda_{l}'},\chi_r \rangle.
\end{multline*}
This expression equals to $1$ if  $l_j=\#\{\lambda_i=j\}$ for every $j\ge 1$ and equals to zero otherwise, so the first claim follows.

Lemma~\ref{lm:truncated_limit} together with Theorem~\ref{thm:homological_grassmannian} yield
\begin{multline*}
A_{*,*}(\HGr_+)=\varinjlim A_{*,*}(\HGr(2r,2n))=\\
=\varinjlim \Hom_{A^{*,*}(\pt)}(A^{*,*}(\HGr(2r,2n)), A_{*,*}(\pt)).
\end{multline*}
We have an explicit computation of $A^{*,*}(\HGr(2r,2n))$ given by Theorem~\ref{thm:cohomological_HGr_finite}, so the second claim follows from the first one.
\end{proof}

\section{Preliminaries on $\KO$}

In this section we gather the representability results for higher Grothendieck-Witt groups (also known as hermitian $\K$-theory) and fix a symplectic orientation on it. Recall that the characteristic of the base field is assumed to be different from $2$.

\begin{definition}
Let $X$ be a smooth variety and $U\subset X$ be an open subset. For $n,i\in \Z$ denote $\GW^{[n]}_i(X,U)$ higher Grothendieck-Witt groups defined by Schlichting \cite[Definition~8]{Sch10b}, see also \cite{Sch10a, Sch12}. Recall that by \cite[Proposition~6.3]{Sch12} (cf. \cite[Theorem~2.4]{Wal03}) for $i<0$ there is a canonical identification $\GW^{[n]}_i(X,U)\cong \W^{[n-i]}(X,U)$, where the latter groups are derived Witt groups defined by Balmer \cite{Bal99}. Moreover, $\GW^{[0]}_0(X)$ and $\GW^{[2]}_0(X)$ coincide with the Grothendieck-Witt group of $X$ introduced by Knebusch \cite{Kne77} and its symplectic version respectively. 

For an orthogonal (resp. symplectic) bundle $E$ over a smooth variety $X$ we denote 
\begin{itemize}
\item
$\langle E \rangle\in \GW^{[0]}_0(X)$ (resp. $\langle E \rangle\in \GW_0^{[2]}(X)$) the corresponding element in the Grothendieck-Witt group,
\item
$[E]\in \W^{[0]}(X)$ (resp. $[E]\in \W^{[2]}(X)$) the corresponding element in the Witt group.
\end{itemize}
\end{definition}

\begin{definition}
We need the following notation complementary to the one introduced in Definition~\ref{def:symplectic_grassmannians} (cf. \cite{PW10c}).
\begin{itemize}
\item
$\Hplane_{+}=\left( k^{\oplus 2}, \begin{pmatrix} 0 & 1 \\ 1 & 0 \end{pmatrix} \right)$ is the \textit{standard hyperbolic plane}.
\item
$\RGr(2r,2n)=\bigslant{\Orth_{2n}}{\Orth_{2r}\times \Orth_{2n-2r}}$ is the \textit{real Grassmannian}. Here the orthogonal groups are taken with respect to the hyperbolic quadratic form $x_1x_2+x_3x_4+\dots+x_{2n-1}x_{2n}$. Similar to the quaternionic case, the real Grassmannian could be described as the open subscheme of $\Gr(2r,\Hplane_{+}^{\oplus n})$ parametrizing subspaces on which the standard hyperbolic quadratic form is nondegenerate. 
\item
$\mathcal{U}^o_{2r,2n}$ is the tautological rank $2r$ orthogonal vector bundle over $\RGr(2r,2n)$.
\item
$\RGr=\varinjlim\limits_{r,n} (\RGr(2r,2n),*)$ is the \textit{infinite real Grassmannian} considered as a pointed motivic space. Here $*=\RGr(2,2)\in \RGr(2r,2n)$.
\end{itemize}
\end{definition}

\begin{theorem}[{\cite[Theorem~1.1]{ST13}, see also \cite[Theorem~8.2]{PW10c}}]\label{thm:GW_unstable_repr}
Let $X$ be a smooth variety and $U$ be an open subset of $X$. Denote $\underline{\Z}$ the sheaf associated to the presheaf $\Z$. Then there are natural isomorphisms
\begin{gather*}
\Hom_{\Ho(k)}(X/U, \underline{\Z}\times \RGr)\cong \GW^{[0]}_0(X,U),\\
\Hom_{\Ho(k)}(X/U, \underline{\Z}\times \HGr)\cong \GW^{[2]}_0(X,U).
\end{gather*}
Under these isomorphisms the tautological morphisms 
\[
\RGr(2r,2n)\to \{m\}\times \RGr,\quad \HGr(2r,2n)\to \{m\}\times \HGr
\]
correspond to 
\begin{gather*}
\langle\mathcal{U}^o_{2r,2n} \rangle+(m-r) \langle \Hplane_+ \rangle \in \GW^{[0]}_0(\RGr(2r,2n)),\,\\ 
\langle \mathcal{U}^s_{2r,2n} \rangle + (m-r) \langle \Hplane_{-}\rangle \in \GW^{[2]}_0(\HGr(2r,2n))
\end{gather*}
respectively.
\end{theorem}

\begin{remark} \label{rem:classes_as_operations}
Let $A$ be a symplectically oriented spectrum. Then this theorem via the Yoneda Lemma allows us to interpret characteristic classes, i.e. elements of $A^{*,*}(\HGr)$, as natural transformations $\GW^{[2]}_0(X)\to A^{*,*}(X)$.
\end{remark}

\begin{definition}
Let $Y$ be a pointed motivic space. Put 
\begin{align*}
&\GW^{[0]}_0(Y)=\Hom_{\Ho(k)}(Y, \underline{\Z}\times \RGr),\\
&\GW^{[2]}_0(Y)=\Hom_{\Ho(k)}(Y, \underline{\Z}\times \HGr).
\end{align*}
For a family of pointed smooth varieties $(X_1,x_1), (X_2,x_2),\dots, (X_m,x_m)$ and $n=0,2$ we identify $\GW^{[n]}_0((X_1,x_1)\wedge (X_2,x_2)\wedge \dots \wedge (X_m,x_m))$ with the subgroup of $\GW^{[n]}_0(X_1\times X_2 \times \dots \times X_m)$ consisting of all the elements $\alpha$ satisfying $\alpha|_{X_1\times\dots \times X_{j-1}\times \{x_j\}\times X_{j+1}\times \dots \times X_m}=0$ for all $j$.
\end{definition}

\begin{definition}
Let $\tau^s\in \GW^{[2]}_0(\HGr)$ and $\tau^o\in \GW^{[0]}_0(\RGr)$ be the tautological elements over the infinite Grassmannians represented by identity morphisms $\HGr\to \{0\}\times \HGr$ and $\RGr\to \{0\}\times \RGr$ and satisfying
\[
\tau^s|_{\HGr(2r,2n)}= \langle\mathcal{U}^s_{2r,2n} \rangle - r \langle \Hplane_{-} \rangle,\, \tau^o|_{\RGr(2r,2n)}= \langle \mathcal{U}^o_{2r,2n} \rangle - r \langle \Hplane_{+} \rangle.
\]
\end{definition}

\begin{definition}\label{def:KO}
The periodic $\HP^1$-spectrum $\KO$ is given by the spaces
\[
\KO=(\RGr, \HGr, \RGr, \HGr, \hdots)
\]
and structure maps 
\[
\sigma_{\KO}^{o}\colon \HP^1\wedge \RGr \to \HGr,\quad \sigma_{\KO}^{s}\colon \HP^1\wedge \HGr \to \RGr
\]
satisfying 
\begin{gather*}
(\sigma_{\KO}^{o})^{\GW}(\tau^s)|_{\HP^1\wedge \RGr(2r,2n)} =(\langle \Hstruct(1) \rangle - \langle \Hplane_{-} \rangle) \boxtimes \tau^o|_{\RGr(2r,2n)},\\ 
(\sigma_{\KO}^{s})^{\GW}(\tau^o)|_{\HP^1\wedge \HGr(2r,2n)} =(\langle \Hstruct(1) \rangle - \langle \Hplane_{-} \rangle) \boxtimes \tau^s|_{\HGr(2r,2n)}.
\end{gather*}
Here $\boxtimes$ is induced by the external tensor product of vector bundles, 
\[
E_1\boxtimes E_2 = p_1^*E\otimes p_2^* E_2
\]
for vector bundles $E_1$ and $E_2$ over $X_1$ and $X_2$ respectively with projections $p_i\colon X_1\times X_2\to X_i$. Note that (external) tensor product of two symplectic vector bundles has a canonical orthogonal structure, while (external) tensor product of a symplectic and an orthogonal bundle is symplectic.

The above morphisms $\sigma_{\KO}^{o}$ and $\sigma_{\KO}^{s}$ exist as morphisms of pointed sheaves by \cite[Proposition 12.4, Lemma 12.5,12.6]{PW10c}. This defined spectrum is canonically isomorphic in $\SH(k)$ to the spectra $\mathbf{BO}^{geom}$ and $\mathbf{BO}$ constructed in \cite{PW10c}.
\end{definition}

\begin{theorem}[{\cite[Theorems~1.3,~1.5]{PW10c}}] \label{thm:stable_representability}
The spectrum $\KO$ can be endowed with the structure of a commutative ring spectrum $(\KO,m_{\KO},u_{\KO})$. Moreover, this commutative ring spectrum represents higher Grothendieck-Witt groups, i.e. for every smooth variety $X$ and an open subset $U\subset X$ there exist canonical functorial isomorphisms
\[
\Theta\colon \KO^{i,j}(X/U)\xrightarrow{\simeq} \GW^{[j]}_{2j-i}(X,U)
\]
satisfying
\begin{enumerate}
\item
$\Theta$ commutes with the connecting homomorphisms $\partial$ in localization sequences,
\item
the $\cup$-product on $\KO^{*,*}(-)$ induced by the monoid structure of $\KO$ agrees with the Gille-Nenashev right pairing \cite[Theorem~2.9]{GN03} lifted to $\GW^{[*]}_0(-)$ as in \cite[\S 4]{PW10c},
\item
$\Theta(1)=1, \Theta(\varepsilon)=\langle -1 \rangle$.
\end{enumerate}
\end{theorem}

\begin{remark}
In view of the above theorem we identify $\KO^{0,0}(X)\cong \GW^{[0]}_0(X)$ and $\KO^{4,2}(X)\cong \GW^{[2]}_0(X)$.
\end{remark}

\begin{theorem}\label{thm:KO_orientation}
The rule which assigns to a rank $2$ symplectic bundle $E$ over a smooth variety $X$ class $b_1^{\KO}(E)=\langle E \rangle - \langle \Hplane_{-} \rangle \in \KO^{4,2}(X)$ can be uniquely extended to a Borel classes theory and by Theorem~\ref{thm:bijection_orientations} induces a symplectic orientation of $\KO$.
\end{theorem}
\begin{proof}
Existence of the Borel classes theory follows from \cite[Theorem~5.1]{PW10c} and uniqueness follows from \cite[Theorem~14.4b]{PW10a}
\end{proof}

The next two lemmas follow immediately from the construction of $\Theta$.
\begin{lemma}\label{lm:stabilization}
Let $X$ be a smooth variety. Then the following diagram commutes.
\[
\xymatrix{
\Hom_{\Ho (k)}(X_+,\HGr) \ar[d]^i \ar[r]^(0.42){\Sigma^\infty_{\HP^1}} & \Hom_{\SH (k)}(\Sigma^\infty_{\HP^1}X_+, \Sigma^\infty_{\HP^1}\HGr) \ar[d]_{\cong}^{\phi} \\
\Hom_{\Ho (k)}(X_+,\underline{\Z}\times \HGr) \ar[dd]^{\cong}_f & \Hom_{\SH (k)}(\Sigma^\infty_{\HP^1}X_+,\Tr_1 \KO\wedge \HP^1) \ar[d]^{j}\\
 & \Hom_{\SH (k)}(\Sigma^\infty_{\HP^1}X_+, \KO\wedge \HP^1)\ar[d]^=\\
\GW^{[2]}_0(X) & \KO^{4,2}(X) \ar[l]_{\Theta}^{\cong}
}
\]
Here 
\begin{itemize}
\item
$i$ is induced by the identity morphism $\HGr\xrightarrow{} \{0\}\times \HGr$, 
\item
$\phi$ is induced by the canonical isomorphisms
\[
\Sigma^\infty_{\HP^1}\HGr \xleftarrow{\simeq} \Sigma^\infty_{\HP^1}\HGr\{-1\}\wedge \HP^1 \xrightarrow{\simeq} \Tr_1 \KO\wedge \HP^1,
\] 
\item
$j$ is induced by the canonical morphism $\Tr_1 \KO\to \KO$, 
\item
$f$ and $\Theta$ are given by Theorem~\ref{thm:GW_unstable_repr} and~\ref{thm:stable_representability} respectively.
\end{itemize}
\end{lemma}

\begin{lemma}\label{lm:unit_KO}
The following diagram commutes. 
\[
\xymatrix @C=7pc {
\Sigma^{\infty}_{\HP^1} \HP^1 \{-1\} \ar[r]^{\phi}_\simeq \ar[dr]_{u_{\KO}'} & \mathbb{S} \ar[d]^{u_{\KO}} \\
&  \KO
}
\]
Here
\begin{itemize}
\item
$u_{\KO}$ is the unit morphism,
\item
$\phi$ is an isomorphism which is identity starting from the first space,
\item
$u_{\KO}'=(f_0,f_1,f_2,\dots)$ with $f_n\colon (\HP^1)^{\wedge n}\to \KO_n$ satisfying
\begin{align*}
&f_{2m-1}^{\GW}(\tau^s)= \underbrace{(\langle \Hstruct(1) \rangle - \langle  \Hplane_{-} \rangle )\boxtimes \dots \boxtimes (\langle \Hstruct(1) \rangle - \langle  \Hplane_{-} \rangle )}_{2m-1}, \\
&f_{2m}^{\GW}(\tau^o)= \underbrace{(\langle \Hstruct(1) \rangle - \langle  \Hplane_{-} \rangle )\boxtimes \dots \boxtimes (\langle \Hstruct(1) \rangle - \langle  \Hplane_{-} \rangle )}_{2m}
\end{align*}
for $n\ge 1$.
\end{itemize}
\end{lemma}

\begin{corollary}\label{cor:normalization}
Let $\Hstruct(1)$ be the tautological rank $2$ symplectic bundle over $\HProj^1$. Then
\begin{enumerate}
\item
$\Sigma_{\HProj^1} 1= b_1^{\KO}(\Hstruct(1)) \in \KO^{4,2}(\HProj^1)$,
\item
$\Sigma_{\HProj^1} 1= \chi_1 \in \KO_{4,2}(\HProj^1)$.
\end{enumerate}
\end{corollary}
\begin{proof}
With our definition $ b_1^{\KO}(\Hstruct(1)) = \langle \Hstruct(1) \rangle - \langle  \Hplane_{-} \rangle $ the first claim is straightforward from the above two lemmas. The second claim follows from the first one since $\langle \Sigma^1_{\HProj^1} 1, \Sigma^1_{\HProj^1} 1\rangle =1$ for the Kronecker product.
\end{proof}

\begin{definition}
The cohomology theory $\KO^{*,*}(-)$ is $(8,4)$-periodic with the periodicity isomorphism induced by
\[
\KO\wedge \Sph^{8,4} \cong \KO\wedge (\HP^1)^{\wedge 2} \xrightarrow{\simeq} \KO\{2\}\cong \KO.
\]
Here the first isomorphism is given by Corollary~\ref{cor:HP1_sphere}, the second isomorphism is the canonical one identifying double $\HP^1$-suspension with shift by $2$ and the third isomorphism is given by identity map.

One may identify this periodicity isomorphisms with 
\[
\KO\wedge \Sph^{8,4} \xrightarrow{-\cup \Sigma^{8,4}\beta} \KO,
\]
where  $\beta\in \KO^{-8,-4}(\pt)$ is the element corresponding to $1\in \KO^{0,0}(\pt)$ under the categorical periodicity isomorphism
\[
\KO^{0,0}(\pt)\cong \GW_{0}^{[0]}(\pt) \cong \GW_{0}^{[-4]}(\pt)\cong \KO^{-8,-4}(\pt),
\]
i.e. $\beta$ is the unique element satisfying 
\[
\Sigma^2_{\HProj^1}\beta = (\langle \Hstruct(1) \rangle - \langle  \Hplane_{-} \rangle )\boxtimes (\langle \Hstruct(1) \rangle - \langle  \Hplane_{-} \rangle ) \in \KO^{0,0}(\HProj^1\wedge \HProj^1).
\]
We refer to $\beta$ as \textit{Bott element}.
\end{definition}

\begin{remark}
For a spectrum $\K$ representing algebraic $\K$-theory there exists a morphism $\KO\xrightarrow{F}\K$ that induces forgetful maps
\[
F\colon \GW^{[0]}_0(X)\cong \KO^{0,0}(X)\to \K^{0,0}(X) \cong \K_0(X).
\]
Recall that $\K$ is $(2,1)$-periodic with the periodicity realized by cup product with the element $\beta_\K\in\K^{-2,-1}(\pt)$ satisfying 
\[
\Sigma_{\mathbb{P}^1}\beta_\K = [\mathcal{O}(-1)]-1 \in \K^{0,0}(\mathbb{P}^1,\infty).
\]
One can show that $F(\beta)=\beta_\K^4$.
\end{remark}

\begin{remark}
Let $E_1,E_2$ be symplectic bundles over a smooth variety $X$. Then
\[
\beta\cup \langle  E_1 \rangle \cup \langle E_2 \rangle = \langle E_1\otimes E_2 \rangle.
\]
Here on the left side we consider $E_1,E_2$ as elements of $\KO^{4,2}(X)$ and on the right side we consider them as symplectic bundles, so $ E_1\otimes E_2$ is an orthogonal bundle which we treat as an element of $\KO^{0,0}(X)$.
\end{remark}

\section{Hopf element and $\KW$}

In this section we recall the definition of the Hopf element and identify $\KO[\eta^{-1}]$ as a spectrum representing derived Witt groups.

\begin{definition} \label{def_hopf}
The \textit{Hopf map} is the projection
\[
H\colon \A^2-\{0\} \to \mathbb{P}^1
\]
given by $H(x,y)=[x,y]$. Pointing $\A^2-\{0\}$ by $(1,1)$ and $\mathbb{P}^1$ by $[1:1]$ and taking the suspension spectra we obtain a morphism
\[
\Sigma^\infty_{\HP^1} H \in \Hom_{\SH (k)}(\Sigma^\infty_{\HP^1} (\A^2-\{0\},(1,1)),\Sigma^\infty_{\HP^1} (\mathbb{P}^1,[1:1])).
\]
The \textit{Hopf element} $\eta=\Sigma^{-3,-2} \Sigma^\infty_{\HP^1} H\in \mathbb{S}^{-1,-1}(\pt)$ is the element corresponding to $\Sigma^\infty_{\HP^1} H$ under the suspension isomorphism and canonical isomorphisms 
\[
(\mathbb{P}^1,[1:1])\cong \Sph^{2,1},\quad (\A^2-\{0\},(1,1))\cong \Sph^{3,2}
\]
given by \cite[Lemma 3.2.15, Corollary 3.2.18 Example 3.2.20]{MV99}.
\end{definition}

\begin{definition}
Denote
\begin{gather*}
\mathbb{S} [\eta^{-1}]=
\operatorname{hocolim} \left(\mathbb{S} \xrightarrow{ \cup \eta} \mathbb{S}\wedge \Sph^{-1,-1} \xrightarrow{\cup \eta} \mathbb{S}\wedge \Sph^{-2,-2} \xrightarrow{ \cup \eta} \dots \right),\\
\KW=\KO\wedge \mathbb{S}[\eta^{-1}].
\end{gather*}
This spectrum inherits the structure of a $(8,4)$-periodic symplectically oriented commutative ring spectrum from $\KO$.
\end{definition}

\begin{remark}
We clearly have 
\[
\KW^{*,*}(\KW)=\KW^{*,*}(\KO),\quad \KW_{*,*}(\KW)=\KW_{*,*}(\KO).
\]
\end{remark}

It is well-known that the spectrum $\KW$ represents derived Witt groups defined by Balmer \cite{Bal99} (see, for example, \cite[Theorem~6.5]{An12}).

\begin{theorem} \label{thm:KW_represents}
For every smooth variety $X$ there exists a functorial in $X$ isomorphism $\Theta_\W\colon \KW^{i,j}(X)\xrightarrow{\simeq} \W^{[i-j]}(X)$ such that the square
\[
\xymatrix{
\KO^{2n,n}(X) \ar[r]^{\Theta}_{\simeq} \ar[d] & \GW^{[n]}_0(X) \ar[d] \\
\KW^{2n,n}(X) \ar[r]^{\Theta_\W}_{\simeq} & \W^{[n]}(X) 
}
\]
commutes for all $n$. Here the left vertical morphism is the canonical one arising from localization and the right vertical morphism is given by killing metabolic elements.
\end{theorem}

\begin{remark}
Having the above theorem in mind we identify $\KW^{0,0}(X)$ and $\KW^{4,2}(X)$ with $\W^{[0]}(X)$ and $\W^{[2]}(X)$ respectively. In particular, we have $b_1^{\KW}(E)=[E]\in \KW^{4,2}(X)$ for a rank $2$ symplectic bundle $E$ over $X$.
\end{remark}

\section{Borel classes of triple tensor product in $\KW$}

In this section in Lemma~\ref{lm:formal_group_law} we compute characteristic classes of a triple tensor product of rank $2$ symplectic bundles. This computation is a derived Witt analogue of equality
\[
c_1^\K(L_1\otimes L_2)=c_1^\K(L_1)+c_1^\K(L_2)- c_1^\K(L_1)c_1^\K(L_2)
\]
in $\K$-theory, where $L_i$ are line bundles and $c_1^\K(L_i)=1-[L_i^\vee]$ is the first Chern class in $\K$-theory. As an intermediate step we show how to express Borel classes in derived Witt groups using external powers. 

\begin{lemma} \label{lm:Borel_via_lambda}
Let $E$ be a symplectic bundle of rank $8$ over a smooth variety $X$. Then
\[
\begin{aligned}
&b_1^{\KW}(E)=[E],\\
&\beta b_2^{\KW}(E)= [\Lambda^2 E]  -  4,\\
\end{aligned}
\quad
\begin{aligned}
&\beta b_3^{\KW}(E)= [\Lambda^3 E]  -3 [E] ,\\
&\beta^2 b_4^{\KW}(E)= [\Lambda^4 E]  -2[\Lambda^2 E]  + 2 .
\end{aligned}
\]
\end{lemma}
\begin{proof}
Using Theorem~\ref{thm:splitting} we may assume that $E=E_1\oplus E_2\oplus E_3\oplus E_4$ for rank $2$ symplectic bundles $E_i$.  Then $\beta^{\left[ \tfrac{n}{2}\right]}b_n^{\KW}(E)=\sigma_n(E_1,E_2,E_3,E_4)$. 

Expanding 
\[
\Lambda^j(E_1\oplus E_2\oplus E_3\oplus E_4)=\bigoplus_{i_1+i_2+i_3+i_4=j} \Lambda^{i_1}E_1\otimes \Lambda^{i_2}E_2\otimes \Lambda^{i_3}E_3\otimes \Lambda^{i_4}E_4
\]
and using the given trivializations $\Lambda^2E_i=\triv_X$ we obtain
\[
\begin{aligned}
&\Lambda^1 E=\sigma_1(E_1,E_2,E_3,E_4),\quad \Lambda^2 E=\sigma_2(E_1,E_2,E_3,E_4)+4,\\
&\Lambda^3 E=\sigma_3(E_1,E_2,E_3,E_4) + 3 \sigma_1(E_1,E_2,E_3,E_4),\\
&\Lambda^4 E= \sigma_4(E_1,E_2,E_3,E_4) + 2 \sigma_2(E_1,E_2,E_3,E_4) +6.
\end{aligned}
\]
The claim follows.
\end{proof}

\begin{lemma}\label{lm:formal_group_law}
Let $E_1,E_2$ and $E_3$ be rank $2$ symplectic bundles over a smooth variety $X$. Put $\xi_i = b_1^{\KW}(E_i)\in \KW^{4,2}(X)$ and denote $\xi(n_1,n_2,n_3)$ the sum of all the monomials lying in the orbit of $\xi_1^{n_1}\xi_2^{n_2}\xi_3^{n_3}$ under the action of $S_3$. Then
\begin{align*}
& b^{\KW}_1(E_1\otimes E_2 \otimes E_3)=\beta \xi(1,1,1),\\
& b^{\KW}_2(E_1\otimes E_2 \otimes E_3)=\beta\xi(2,2,0)-2\xi(2,0,0),\\
& b^{\KW}_3(E_1\otimes E_2 \otimes E_3)=\beta\xi(3,1,1) -8\xi(1,1,1),\\
& b^{\KW}_4(E_1\otimes E_2 \otimes E_3)=\beta\xi(2,2,2)+ \xi(4,0,0)- 2\xi(2,2,0).
\end{align*}
\end{lemma}
\begin{proof}
Consider the representation ring
\begin{multline*}
\mathrm{Rep}(\Sp_2\times \Sp_2\times \Sp_2)\cong \Z[\chi_1^{\pm 1},\chi_2^{\pm 1},\chi_3^{\pm 1}]^{\Z/2\times \Z/2 \times \Z/2}=\\
=\Z[\chi_1+\chi_1^{-1},\chi_2+\chi_2^{-1},\chi_3+\chi_3^{-1}]
\end{multline*}
with the action of the $i$-th copy of $\Z/2$ given by $\chi_i\leftrightarrow \chi_i^{-1}$. The exterior powers of representations give rise to the operations 
\[
\Lambda^m\colon \Z[\chi_1+\chi_1^{-1},\chi_2+\chi_2^{-1},\chi_3+\chi_3^{-1}] \to \Z[\chi_1+\chi_1^{-1},\chi_2+\chi_2^{-1},\chi_3+\chi_3^{-1}], m\in \mathbb{N}_0,
\]
that are compatible with the operations 
\[
\Lambda^m\colon \Z[\chi_1^{\pm 1},\chi_2^{\pm 1},\chi_3^{\pm 1}]\to \Z[\chi_1^{\pm 1},\chi_2^{\pm 1},\chi_3^{\pm 1}],\quad m\in\mathbb{N}_0,
\]
characterized by the following properties:
\begin{enumerate}
\item
$\Lambda^m(0)=0$,
\item
$\Lambda^m(\chi_1^{i_1}\chi_2^{i_2}\chi_3^{i_3})=\left[\begin{array}{ll} 1, & m=0,\\ \chi_1^{i_1}\chi_2^{i_2}\chi_3^{i_3}, & m=1,\\ 0, & \text{otherwise,} \end{array}\right.$
\item
$\Lambda^m(f+g)=\bigoplus\limits_{m_1+m_2=m}(\Lambda^{m_1}f)(\Lambda^{m_2}g)$.
\end{enumerate} 
Set $e_i=\chi_i+\chi_i^{-1}$. A straightforward computation in $\Z[\chi_1^{\pm 1},\chi_2^{\pm 1},\chi_3^{\pm 1}]$ shows that
\begin{align*}
&\Lambda^1 (e_1e_2e_3) =   e_1 e_2 e_3,\\
&\Lambda^2 (e_1e_2e_3) =   e_1^2 e_2^2 + e_1^2 e_3^2 + e_2^2  e_3^2  - 2 (e_1^2+e_2^2+e_3^2)+4,\\
&\Lambda^3 (e_1e_2e_3) =   e_1^3e_2e_3 + e_1e_2^3e_3+  e_1e_2e_3^3  - 5 e_1e_2e_3, \\
&\Lambda^4 (e_1e_2e_3) =   e_1^4+ e_2^4 + e_3 ^4 + e_1^2e_2^2e_3^2 - 4 (e_1^2+e_2^2 + e_3^2) +6.
\end{align*}

Thus
\begin{alignat*}{3}
&\Lambda^1 (E_1\otimes E_2\otimes E_3) &&=   E_1  \otimes E_2 \otimes E_3,\\
&\Lambda^2 (E_1\otimes E_2\otimes E_3) &&=   E_1^2 \otimes E_2^2 + E_1^2 \otimes  E_3^2 + E_2^2 \otimes  E_3^2  - 2 (E_1^2+E_2^2+E_3^2)+4,\\
&\Lambda^3 (E_1\otimes E_2\otimes E_3) &&=   E_1^3\otimes E_2\otimes E_3 + E_1\otimes E_2^3 \otimes E_3+  E_1\otimes E_2\otimes E_3^3 - \\
& & &  - 5 E_1\otimes E_2\otimes E_3,  \\
&\Lambda^4 (E_1\otimes E_2\otimes E_3) &&=   E_1^4+ E_2^4 + E_3 ^4 + E_1^2\otimes E_2^2 \otimes E_3^2 - 4 (E_1^2+E_2^2 + E_3^2) +6.
\end{alignat*}

The claim of the lemma follows by Lemma~\ref{lm:Borel_via_lambda}.
\end{proof}


\section{Stable operations in $\KW_\Q$}

In this section we compute the algebra of stable operations in $\KW_\Q$, i.e. $\KW_\Q^{*,*}(\KW_\Q)$. The computation is straightforward and based on Lemma~\ref{lm:truncated_limit} combined with Theorem~\ref{thm:cohomological_HGr}.

\begin{lemma}\label{lm:main_computation}
Let $B\in \Hom_{\Ho(k)}(\HProj^1\wedge \HProj^1 \wedge \HGr, \HGr)$ be the morphism
characterized by the property 
\[
B^{\GW}(\tau^s)=(\langle \Hstruct(1) \rangle - \langle \Hplane_{-} \rangle) \boxtimes (\langle \Hstruct(1) \rangle - \langle \Hplane_{-} \rangle)\boxtimes \tau^s.
\]
Then
\[
B^{\KW}(s_i^{\KW}(\tau^s))= [\Hstruct(1) \boxtimes \Hstruct(1)] \cup  ( a_i s^{\KW}_i(\tau^s) + c_i s^{\KW}_{i-2}(\tau^s) )
\]
for
\[
a_{2j+1}=(2j+1)^2,\quad c_{2j+1}=-\beta^{-1}8j(2j+1),\quad a_{2j}=c_{2j}=0.
\]
\end{lemma}
\begin{proof}
By Remark~\ref{rem:classes_as_operations} we may interpret $s_i^{\KW}$ as a natural transformation $\GW^{[2]}_0\to \KW^{4n,2n}$, whence
\[
B^{\KW}(s_i^{\KW}(\tau^s))=s_i^{\KW}(B^{\GW}(\tau^s)).
\]
Thus we need to compute $s_i^{\KW}((\langle \Hstruct(1) \rangle - \langle \Hplane_{-} \rangle)\boxtimes (\langle \Hstruct(1) \rangle - \langle \Hplane_{-} \rangle)\boxtimes \tau^s)$. The classes $s_i^{\KW}$ are additive and $s_i^{\KW}(\langle \Hplane_{-} \rangle )=0$, so it is sufficient to show that
\begin{multline*}
s_i^{\KW}((\langle \Hstruct(1) \rangle - \langle \Hplane_{-} \rangle)\boxtimes (\langle \Hstruct(1) \rangle - \langle \Hplane_{-} \rangle)\boxtimes  (\langle \Hstruct(1) \rangle - \langle \Hplane_{-} \rangle))
=\\
=
[\Hstruct(1)\boxtimes \Hstruct(1)]\cup  ( a_i s^{\KW}_i(\langle \Hstruct(1) \rangle ) + c_i s^{\KW}_{i-2}(\langle \Hstruct(1) \rangle) )
\end{multline*}
for 
\begin{multline*}
(\langle \Hstruct(1) \rangle - \langle \Hplane_{-} \rangle)\boxtimes  (\langle \Hstruct(1) \rangle - \langle \Hplane_{-} \rangle) \boxtimes  (\langle \Hstruct(1) \rangle - \langle \Hplane_{-} \rangle)\in \\ \in \GW^{[2]}_0(\HProj^1\wedge \HProj^1\wedge \HProj^\infty).
\end{multline*}
Denote
\begin{gather*}
x= b_1^{\KW}(\Hstruct(1)\boxtimes \triv\boxtimes\triv), \, y= b_1^{\KW}(\triv\boxtimes \Hstruct(1)\boxtimes\triv),\, \xi= b_1^{\KW}(\triv\boxtimes \triv\boxtimes \Hstruct(1)),\\
b_t(x,y,\xi)=b^{\KW}_t(\Hstruct(1)\boxtimes \Hstruct(1)\boxtimes \Hstruct(1)),\\ 
s_t(x,y,\xi)=s^{\KW}_t(\Hstruct(1)\boxtimes \Hstruct(1)\boxtimes \Hstruct(1)).
\end{gather*}
In this notation the claim is equivalent to
\begin{multline*}
s_t(x,y,\xi)-s_t(0,y,\xi)-s_t(x,0,\xi)-s_t(x,y,0)+\\
+s_t(0,0,\xi)+s_t(0,y,0)+s_t(x,0,0)-s_t(0,0,0)=\\
= \beta  x y \sum_{i\ge 1} ( a_i \xi^i + c_i \xi^{i-2} ) t^i.
\end{multline*}
The main summand on the left side is $s_t(x,y,\xi)$ and the other summands just cancel from $s_t(x,y,\xi)$ all the terms that do not contain $xy\xi$. Since $x^2=y^2=0$ Lemma~\ref{lm:formal_group_law} yields
\[
b_t(x,y,\xi)=1+\beta xy\xi t -2 \xi^2t^2 + (\beta xy\xi^3-8xy\xi)t^3+ \xi^4t^4.
\]
Thus
\begin{multline*}
s_{t}(x,y,\xi)=-t \frac{d}{dt}\ln b_{-t}(x,y,\xi)=\\
=-t\frac{\tfrac{d}{dt}\left( (1 - \xi^2t^2)^2 - xy\xi (\beta t+(\beta \xi^2-8)t^3)\right)}{(1 - \xi^2t^2)^2 - xy\xi (\beta t+(\beta \xi^2-8)t^3)} .
\end{multline*}
Put
\[
A(\xi,t)=(1 - \xi^2t^2)^2,\quad B(\xi,t)=\xi (\beta t+(\beta \xi^2-8)t^3).
\]
Recall that $x^2=y^2=0$ whence $(xy)^2=0$ and 

\begin{multline*}
s_{t}(x,y,\xi)=-t\frac{\tfrac{d}{dt} (A(\xi,t)-xyB(\xi,t))}{A(\xi,t)-xyB(\xi,t)}=\\
= -t\frac{\left(\tfrac{d}{dt} (A(\xi,t)-xyB(\xi,t))\right) \left(A(\xi,t)+xyB(\xi,t)\right)}{A(\xi,t)^2}.
\end{multline*}

Expanding the numerator, applying $x^2y^2=0$ and omitting all the terms that do not contain $xy\xi$ we obtain
\begin{multline*}
s_{t}(x,y,\xi)=-t\frac{\left(\tfrac{d}{dt} A(\xi,t)\right)xyB(\xi,t) - A(\xi,t)\tfrac{d}{dt} \left( xyB(\xi,t)\right)}{A(\xi,t)^2}=\\
=xyt\frac{d}{dt}\left( \frac{B(\xi,t)}{A(\xi,t)} \right)
=\beta xy t\frac{d}{dt}\left( \frac{\xi  t + (\xi^3-8\beta^{-1} \xi) t^3}{(1-\xi^2t^2)^2} \right) =\\
= \beta xy t \frac{d}{dt} \left((\xi  t + (\xi^3-8\beta^{-1} \xi) t^3) (\sum_{j\ge 0} (j+1)\xi^{2j}t^{2j} )\right) =\\
= \beta xy t \frac{d}{dt} \left(\sum_{j\ge 0} ((2j+1)\xi^{2j+1}-8\beta^{-1} j\xi^{2j-1})t^{2j+1} \right)=\\
= \beta xy \sum_{j\ge 0} ((2j+1)^2\xi^{2j+1}-8\beta^{-1} j(2j+1)\xi^{2j-1})t^{2j+1}.
\end{multline*}
\end{proof}

\begin{lemma} \label{lm:coh_shift_by_beta}
The following diagram commutes.
\[
\xymatrix{
\KW_\Q^{*,*}(\KO) \ar[d]^{R } \ar[r]^(0.65){} & \varprojlim \KW_\Q^{*+8n+4,*+4n+2}(\HGr)\ar[d]^T \\
\KW_\Q^{*+8,*+4}(\KO) \ar[r]^(0.65){}  & \varprojlim \KW_\Q^{*+8(n+1)+4,*+4(n+1)+2}(\HGr) 
}
\]
Here the horizontal homomorphisms are the canonical ones given by Lemma~\ref{lm:truncated_limit}, $T$ is induced by the shift
\begin{align*}
\prod_{n\ge 0} {\KW_\Q}^{*+8n+4,*+4n+2}(\HGr) & \to\prod_{n\ge 0} {\KW_\Q}^{*+8(n+1)+4,*+4(n+1)+2}(\HGr)\\
(t_1,t_3,t_5,\dots) &\mapsto (t_3,t_5,\dots)
\end{align*}
and $R$ is given by $R(\gamma)=(\Sigma^{8,4} \gamma ) \circ (-\cup \beta^{-1})$.
\end{lemma}
\begin{proof}
Straightforward from the commutativity of the following diagram:
\[
\xymatrix{
\Tr_{2n+1} \KO \ar[d]^i \ar[rr]^= & & (\Tr_{2(n+1)+1} \KO) \{2\} \ar[d]^i \\
\KO \ar[r]^{ - \cup \beta^{-1}} & \KO\wedge \Sph^{8,4} \ar[r]^\simeq & \KO\{2\} 
}
\]
\end{proof}

\begin{lemma} \label{lm:stable_as_sequence}
Let $\gamma\in \KW_\Q^{0,0}(\KO)$ be a stable operation such that
\[
\gamma \mapsto
(\gamma_1,\gamma_3,\dots)\in  \varprojlim \KW_\Q^{8n+4,4n+2}(\HGr)
\]
under the canonical morphism given by Lemma~\ref{lm:truncated_limit}. Let $X$ be a pointed motivic space and let 
\[
f=(f_0,f_1,f_2,\dots)\colon \Sigma^{\infty}_{\HP^1} X \{-1\} \to \KO
\]
be a morphism of spectra. Then
\[
\gamma(f)=\Sigma_{\HP^1}^{-1}\gamma_{1}(f_{1}),
\]
where $f_{1}\in \Hom_{\Ho (k)}(X,\HGr)$ is treated as an element of $\GW^{[2]}_0(X)$ and $\gamma_{1}$ is treated as an operation $\GW^{[2]}_0\to \KW_\Q^{4,2}$.
\end{lemma}
\begin{proof}

Follows from Lemma~\ref{lm:stabilization}.

\end{proof}

\begin{theorem} \label{thm:rational_stable_operations_KW}
The homomorphism of left $\KW_\Q^{0,0}(\pt)\cong \W_\Q(k)$-modules
\[
Ev\colon \KW_\Q^{0,0}(\KW_\Q) \to \prod_{m\in \Z} \W_\Q(k)
\]
given by 
\[
Ev(\phi)=\left(\dots,\beta^2\phi(\beta^{-2}),\beta\phi(\beta^{-1}),\phi(1),\beta^{-1}\phi(\beta),\beta^{-2}\phi(\beta^{2}),\dots \right)
\]
is an isomorphism of algebras. Here the product on the left is given by composition and the product on the right is the component-wise one. 

Moreover, $\KW_\Q^{p,q}(\KW_\Q)=0$ when $4\nmid p-q$ and the above isomorphism induces an isomorphism of left $\KW_\Q^{*,*}(\pt)\cong \W_\Q(k)[\eta^{\pm 1},\beta^{\pm 1}]$-modules
\[
\KW_\Q^{*,*}(\KW_\Q)\cong \bigoplus_{r,s\in\Z} \beta^r\eta^s\prod_{m\in \Z} \W_\Q(k)
\]
with $\deg \beta =(-8,-4)$, $\deg \eta =(-1,-1)$.
\end{theorem}
\begin{proof}
Having in mind the canonical identifications
\[
\KW_\Q^{*,*}(\KW_\Q)=\KW_\Q^{*,*}(\KO_\Q)=\KW_\Q^{*,*}(\KO)
\]
we focus on the computation of $\KW_\Q^{*,*}(\KO)$.

Lemma~\ref{lm:truncated_limit} yields the short exact sequence
\[
\resizebox{\textwidth}{!}{$
0\to \varprojlim\nolimits^1 \KW^{*+8n+3,*+4n+2}_\Q (\HGr) \to \KW^{*,*}_\Q (\KO) \to \varprojlim_{} \KW_\Q^{*+8n+4,*+4n+2}(\HGr) \to 0
$}
\]
with the limit taken with respect to the morphisms
\[
\Sigma_{\HP^1}^{-2} \circ B^{\KW}\colon \KW_\Q^{*+8n+12,*+4n+6}(\HGr) \to \KW_\Q^{*+8n+4,*+4n+2}(\HGr),
\]
where $B=\sigma_{\KO}^o \circ (\id_{\HP^1}\wedge \sigma_{\KO}^s)$ is the same morphism as in Lemma~\ref{lm:main_computation} up to the canonical identification $\HP^1\cong \HProj^1$. 

Consider the following diagram.
\[
\xymatrix@C=3pc {
 \KW_\Q^{*+8n+12,*+4n+6}(\HGr) \ar[r]^{\pi} \ar[d]^(0.5){B^{\KW}} &     \IQ_\Q^{*+8n+12,*+4n+6}(\HGr)  \ar[dd]^{S_\Q} \ar@/^1.5pc/[ddl]_(0.4){S'_\Q}\\
 \KW_\Q^{*+8n+12,*+4n+6}(\HP^1\wedge \HP^1 \wedge \HGr)  \ar[d]^(0.5){\Sigma_{\HP^1}^{-2} }& \\
 \KW_\Q^{*+8n+4,*+4n+2}(\HGr)  \ar[r]^\pi &  \IQ_\Q^{*+8n+4,*+4n+2}(\HGr)  
}
\]
Here 
\begin{itemize}
\item
$\IQ^{*,*}_\Q(\HGr)=\varprojlim\limits_{m,n} \IQ(\KW_\Q^{*,*}(\HGr(2m,2n),*))$ is the indecomposable quotient (i.e. the ring modulo the reducible elements) of $\KW_\Q^{*,*}(\HGr)$. The Newton identities yield 
\[
(-1)^{i+1}i\overline{b_i^{\KW}(\tau^s)}=\overline{s_i^{\KW}(\tau^s)}
\]
in the indecomposable quotient, thus Theorem~\ref{thm:cohomological_HGr} allows us to identify
\[
\IQ^{*,*}_\Q(\HGr)=\big(\prod_{i\ge 1} \KW^{*-4i,*-2i}_\Q(\pt)\overline{b_i^{\KW}(\tau^s)}\big)_h=\big(\prod_{i\ge 1} \KW^{*-4i,*-2i}_\Q(\pt)\overline{s}_i\big)_h
\]
for $s_i=s_i^{\KW}(\tau^s)$.
\item
$\pi$ is the canonical projection.
\item
$S'_\Q$ is given by $S'_\Q(\overline{s}_i)= \beta a_i s_i+ c_i s_{i-2}$ with 
\[
a_{2j}=c_{2j}=0, a_{2j+1}=(2j+1)^2, c_{2j+1}=-8j(2j+1).
\]
\item
$S_\Q=\pi\circ S'_\Q$.
\end{itemize}
The ring $\KW_\Q^{*,*}(\HP^1\wedge \HP^1 \wedge \HGr)$ has trivial multiplication by Theorem~\ref{thm:cohomological_HGr_finite} (since $b_1^{\KW}(\Hstruct(1))^2=0$ on $\HP^1$), thus $B^{\KW}$ factors through the indecomposable quotient and Lemma~\ref{lm:main_computation} yields commutativity of the diagram. It follows that the canonical homomorphisms
\begin{gather*}
\pi\colon \varprojlim \KW_\Q^{*+8n+4,*+4n+2}(\HGr) \xrightarrow{\simeq} \varprojlim \IQ_\Q^{*+8n+4,*+4n+2}(\HGr), \\
\pi^1\colon \varprojlim\nolimits^1 \KW_\Q^{*+8n+3,*+4n+2}(\HGr) \xrightarrow{\simeq} \varprojlim\nolimits^1 \IQ_\Q^{*+8n+3,*+4n+2}(\HGr) 
\end{gather*}
are isomorphisms.

The morphism
\[
\resizebox{\textwidth}{!}{$
S_\Q\colon \big(\prod\limits_{i\ge 1} \KW^{*+8n-4i+12,*+4n-2i+6}_\Q(\pt)\overline{s}_i \big)_h\to \big(\prod\limits_{i\ge 1} \KW^{*+8n-4i+4,*+4n-2i+2}_\Q(\pt)\overline{s}_i \big)_h
$}
\]
is given by the following matrix:
\[
\begin{pmatrix}
\beta a_1 & 0 & c_3 & 0 & 0 & 0 & \hdots \\
0 & 0 & 0 & 0 & 0 & 0 & \hdots \\
0 & 0 & \beta a_3 & 0 & c_5 & 0 & \hdots \\
0 & 0 & 0 & 0 & 0 & 0 & \hdots \\
0 & 0 & 0 & 0 & \beta a_5 & 0 & \hdots \\
0 & 0 & 0 & 0 & 0 & 0 & \hdots \\
\vdots & \vdots & \vdots & \vdots & \vdots & \vdots & \ddots 
\end{pmatrix},
\]
where $a_{2j+1}$ and $c_{2j+1}$ are invertible. Clearly we have 
\[
\operatorname{Im} (S_\Q\circ S_\Q)=\operatorname{Im} (S_\Q)= \big(\prod_{j\ge 0} \KW^{*+8(n-j),*+4(n-j)}_\Q(\pt)\overline{s}_{2j+1}\big)_h,
\]
thus the $\varprojlim\nolimits^1$ term vanishes. For $4\nmid p-q$ we have
\[
\KW_\Q^{p+8(n-j),q+4(n-j)}(\pt)\cong \W^{[p-q+4(n-j)]}_\Q(k)=0
\]
whence the limit is trivial and $\KW_\Q^{p,q}(\KW_\Q)=0$. In view of the periodicities given by $\eta$ and $\beta$ from now on we deal with $\KW_\Q^{0,0}(\KW_\Q)$. Moreover, it is sufficient to show that the homomorphism $Ev$ from the statement of the theorem is an isomorphism since it clearly agrees with the products.

In order to compute the above limit for $S_\Q$-s we may drop all the terms involving $\overline{s}_{2j}$ and consider 
\[
S_\Q\colon \prod\limits_{j\ge 0} \KW^{8(n-j),4(n-j)}_\Q(\pt)\overline{s}_{2j+1} \to \prod\limits_{j\ge 0} \KW^{8(n-j),4(n-j)}_\Q(\pt)\overline{s}_{2j+1}.
\]
For every $j\ge 0$ choose
\[
\overline{\rho}_{2j+1}=\sum_{l\ge j} \alpha_{2j+1,2l+1} \overline{s}_{2l+1} \in  \prod_{j\ge 0} \KW^{-8j,-4j}_\Q(\pt)\overline{s}_{2j+1}
\]
such that 
\begin{enumerate}
\item
$S_\Q(\overline{\rho}_1)=0$,
\item
$S_\Q(\overline{\rho}_{2j+1})=\beta \overline{\rho}_{2j-1}$, 
\item
$\alpha_{1,1}=1$.
\end{enumerate}
The kernel of $S_\Q$ is a free module of rank $1$ thus (1) and (3) uniquely determine $\overline{\rho}_1$. Item (2) together with the condition that the sum for $\overline{\rho}_{2j+1}$ does not contain  $\overline{s}_1$ uniquely determines $\overline{\rho}_{2j+1}$. One can easily see that $\alpha_{2j+1,2j+1}$ is invertible for every $j$ whence 
\[
\prod_{j\ge 0} \KW^{8(n-j),4(n-j)}_\Q(\pt)\overline{s}_{2j+1}=\prod_{j\ge 0} \KW^{*+8n,*+4n}_\Q(\pt)\overline{\rho}_{2j+1}.
\]
In the new basis consisting of $\overline{\rho}_{2j+1}$-s the morphism $S_\Q$ is just a shift multiplied by $\beta $, thus we can easily compute the inverse limit, obtaining
\[
\varprojlim \KW_\Q^{8n+4,4n+2}(\HGr)=\varprojlim \IQ_\Q^{8n+4,4n+2}(\HGr)= \prod_{m\in \Z} \KW^{0,0}_\Q(\pt)\rho^{st}_{m},
\]
where $\deg \rho^{st}_{m}=(0,0)$ and the structure morphisms
\[
\prod_{m\in \Z} \KW^{0,0}_\Q(\pt)\rho^{st}_{m} \to \KW_\Q^{8n+4,4n+2}(\HGr)
\]
are given by 
\[
\rho^{st}_{m}\mapsto 
\left[
\begin{array}{ll}
\beta^{-n}\rho_{2(m+n)+1}, & m+n\ge 0\\
0, & m+n<0
\end{array}
\right.
\]
for $\rho_{2(m+n)+1}=\sum\limits_{l\ge m+n} \alpha_{2(m+n)+1,2l+1} s_{2l+1} \in \KW_\Q^{4,2}(\HGr)$.

In order to obtain the claim of the theorem it is sufficient to check that 
\[
\beta^{-n}\rho^{st}_{m}(\beta^{n})=
\left[
\begin{array}{ll}
1, & n=m,\\
0, & n\neq m.
\end{array}
\right.
\]
It follows from Lemma~\ref{lm:coh_shift_by_beta} that $\rho^{st}_{m}(\beta^n)=\beta^{n} \rho^{st}_{m-n} (1)$, so it is sufficient to check that
\[
\rho^{st}_{m}(1)=
\left[
\begin{array}{ll}
1, & m=0,\\
0, & m\neq 0.
\end{array}
\right.
\]
Lemma~\ref{lm:stable_as_sequence} yields
\[
\rho^{st}_{m}(1)=\Sigma^{-1}_{\HP^1} \rho_{2m+1}(\langle \Hstruct(1) \rangle -  \langle \Hplane_{-} \rangle).
\]
By the definition of $\rho_{2m+1}$ we have
\[
\rho_{2m+1}(\langle \Hstruct(1) \rangle -  \langle \Hplane_{-} \rangle)=
\left[
\begin{array}{ll}
\sum_{l\ge m} \alpha_{2m+1,2l+1} s_{2l+1}^{\KW}(\langle \Hstruct(1) \rangle -  \langle \Hplane_{-} \rangle), & m\ge 0,\\
0, & m<0.
\end{array}
\right.
\]
All the higher characteristic classes of $\langle \Hstruct(1) \rangle -  \langle \Hplane_{-} \rangle$ vanish while 
\[
s_1^{\KW}(\langle \Hstruct(1) \rangle -  \langle \Hplane_{-} \rangle)=[\Hstruct(1)],
\]
thus
\[
\rho_{2m+1}(\langle \Hstruct(1) \rangle -  \langle \Hplane_{-} \rangle)=
\left[
\begin{array}{ll}
[\Hstruct(1)]=\Sigma_{\HP^1}^1 1, & m=0,\\
0, & m\neq 0
\end{array}
\right.
\]
and the claim follows.
\end{proof}

\begin{remark}
One can restate Theorem~\ref{thm:rational_stable_operations_KW} as follows.  Let 
\[
\mathcal{B}=\left( \Sigma^{8m,4m}\beta^m \right)_{m\in \Z}  \colon \bigoplus_{m\in \Z}  \mathbb{S} \wedge \Sph^{8m,4m} \to \KW_\Q
\]
be the morphism induced by $ \Sigma^{8m,4m} \beta^m\colon \mathbb{S}\wedge \Sph^{8m,4m} \to \KW_\Q$. Then the pullback homomorphism
\[
\mathcal{B}^{\KW_\Q}\colon \KW_\Q^{*,*}(\KW_\Q) \to \KW_\Q^{*,*}(\bigoplus_{m\in \Z} \mathbb{S}\wedge \Sph^{8m,4m})
\]
is an isomorphism.
\end{remark}

\section{Stable cooperations in $\KW_\Q$ and $\KW$}

In this section we compute the algebra of cooperations in $\KW_\Q$ and give an additive description of the cooperations in $\KW$. The approach is dual to the one used in the proof of Theorem~\ref{thm:rational_stable_operations_KW} and based on Lemma~\ref{lm:truncated_limit} and Theorem~\ref{thm:homological_HGr}.

\begin{lemma} \label{lm:shift_by_beta}
The following diagram commutes.
\[
\xymatrix{
\varinjlim (\KW_\Q)_{*+8n+4,*+4n+2}(\HGr) \ar[r]^(0.65){\cong}\ar[d]^T & (\KW_\Q)_{*,*}(\KO) \ar[d]^{- \cprod \beta_r } \\
\varinjlim (\KW_\Q)_{*+8(n+1)+4,*+4(n+1)+2}(\HGr) \ar[r]^(0.65){\cong} & (\KW_\Q)_{*+8,*+4}(\KO)
}
\]
Here the horizontal isomorphisms are the canonical ones given by Lemma~\ref{lm:truncated_limit}, $T$ is induced by the shift
\begin{align*}
\bigoplus_{n\ge 0} (\KW_\Q)_{*+8n+4,*+4n+2}(\HGr) & \to\bigoplus_{n\ge 0} (\KW_\Q)_{*+8(n+1)+4,*+4(n+1)+2}(\HGr)\\
(t_1,t_3,t_5,\dots) &\mapsto (t_3,t_5,\dots),
\end{align*}
$\beta_r=u_{\KW_\Q}\wedge \Sigma^{8,4}\beta\in (\KW_\Q)_{8,4}(\KO)$ and $- \cprod \beta_r $ is given by Definition~\ref{def:right_cap}.
\end{lemma}
\begin{proof}
Follows from the commutativity of the following diagram:
\[
\xymatrix{
\Tr_{2n+1} \KO \ar[d]^i \ar[rr]^= & & (\Tr_{2(n+1)+1} \KO) \{-2\} \ar[d]^i \\
\KO \ar[r]^(0.4){- \cup \beta} & \KO\wedge \Sph^{-8,-4} \ar[r]^\simeq & \KO\{-2\} 
}
\]

\end{proof}

\begin{theorem} \label{thm:cooperations}
Let $u_{\KW_\Q}\colon \SSph\to \KW_\Q$ be the unit map. Then the homomorphism of $\W_\Q(k)[\eta^{\pm 1}]\cong \bigoplus\limits_{n\in\Z}\KW_\Q^{n,n}(\pt)$-algebras
\[
\W_\Q(k)[\eta^{\pm 1}][\beta^{\pm 1}_l,\beta_r^{\pm 1}]\to (\KW_\Q)_{*,*}(\KW_\Q)
\]
given by 
\[
\beta_l\mapsto \Sigma^{8,4}\beta\wedge u_{\KW_\Q},\quad \beta_r\mapsto u_{\KW_\Q}\wedge \Sigma^{8,4}\beta
\]
is an isomorphism. Here the product on the right is given by Definition~\ref{def:right_cap}.
\end{theorem}

\begin{proof}
Abusing the notation put $\beta_l=\Sigma^{8,4}\beta\wedge u_{\KW_\Q}$, $\beta_r=u_{\KW_\Q}\wedge \Sigma^{8,4}\beta$. We need to show that
\[
(\KW_\Q)_{*,*}(\KW_\Q) = \bigoplus_{n,p,q\in \Z} \KW_\Q^{n,n}(\pt) \beta_l^p\cprod \beta_r^q.
\]

Identifying $(\KW_\Q)_{*,*}(\KW_\Q)=(\KW_\Q)_{*,*}(\KO)$ and applying the reasoning dual to the one used in the proof of Theorem~\ref{thm:rational_stable_operations_KW} we obtain that
\[
(\KW_\Q)_{*,*}(\KW_\Q)=\varinjlim (\PE_\Q)_{*+8n+4,*+4n+2} (\HGr),
\]
where 
\[
(\PE_\Q)_{*,*}(\HGr)=\bigoplus_{i\ge 1} (\KW_\Q)_{*-4i,*-2i}(\pt)\overline{s}^\vee_i
\]
is the subspace of $(\KW_\Q)_{*,*}(\HGr)$ dual to $\IQ_\Q^{*,*}(\HGr)$ (see Theorem~\ref{thm:homological_HGr}). Here $\overline{s}_i^\vee\in \PE_{4i,2i}(\HGr)$ satisfies $\langle \overline{s}_i,\overline{s}_i^\vee \rangle = 1$ and $\langle \overline{s}_l,\overline{s}_i^\vee \rangle=0$ for $l\neq i$. The limit is taken with respect to the morphisms
\[
\resizebox{\textwidth}{!}{$
S^\vee_\Q \colon \bigoplus\limits_{i\ge 1} (\KW_\Q)_{*+8n-4i+4,*+4n-2i+2}(\pt)\overline{s}^\vee_i \to \bigoplus\limits_{i\ge 1} (\KW_\Q)_{*+8n-4i+12,*+4n-2i+6}(\pt)\overline{s}^\vee_i
$}
\]
given by $S^\vee_\Q(\overline{s}^\vee_i)=\beta a_i\overline{s}^\vee_i+c_{i+2}\overline{s}^\vee_{i+2}$ for 
\[
a_{2j}=c_{2j}=0,\quad a_{2j+1}=(2j+1)^2,\quad c_{2j+1}=-8j(2j+1)
\]
just as in the proof of Theorem~\ref{thm:rational_stable_operations_KW}. The matrix of $S^\vee_\Q$ is the following one.
\[
\begin{pmatrix}
\beta a_1 & 0 & 0 & 0 & 0 & 0 & \hdots \\
0 & 0 & 0 & 0 & 0 & 0 & \hdots \\
c_3 & 0 & \beta a_3 & 0 & 0 & 0 & \hdots \\
0 & 0 & 0 & 0 & 0 & 0 & \hdots \\
0 & 0 & c_5 & 0 & \beta a_5 & 0 & \hdots \\
0 & 0 & 0 & 0 & 0 & 0 & \hdots \\
\vdots & \vdots & \vdots & \vdots & \vdots & \vdots & \ddots 
\end{pmatrix}.
\]
We can drop all the terms involving $\overline{s}^\vee_{2j}$ obtaining
\[
(\KW_\Q)_{*,*}(\KW_\Q)=\varinjlim_n \bigoplus_{j\ge 0} (\KW_\Q)_{*+8(n-j),*+4(n-j)}(\pt)\overline{s}^\vee_{2j+1}.
\]
For $4\nmid p-q$ the group $(\KW_\Q)_{p,q}(\KW_\Q)$ vanishes and in view of the periodicity realized by cap product with $\eta$ and cap product with $\beta$ (that coincides with multiplication by $\beta_l$, see Definition~\ref{def:right_cap}) from now on we deal with $(\KW_\Q)_{0,0}(\KW_\Q)$.

Denote $\tau_{1}=\overline{s}^\vee_1$, $\tau_{2j+1}=\beta^{-1} S^\vee_\Q(\tau_{2j-1})$. One can easily check that
\[
\bigoplus_{j\ge 0} (\KW_\Q)_{8(n-j),4(n-j)}(\pt)\overline{s}^\vee_{2j+1}=\bigoplus_{j\ge 0} (\KW_\Q)_{8n,4n}(\pt)\tau_{2j+1}.
\]
In this basis $S^\vee_\Q$ is a shift composed with multiplication by $\beta$, hence the limit is easily computed:
\[
\varinjlim_n \bigoplus_{j\ge 0} (\KW_\Q)_{8n,4n}(\pt)\tau_{2j+1}=\bigoplus_{m\in \Z}(\KW_\Q)_{0,0}(\pt) \tau^{st}_{m}
\]
with the structure morphisms 
\[
\bigoplus_{j\ge 0} (\KW_\Q)_{8n,4n}(\pt)\tau_{2j+1}\to \bigoplus_{m\in \Z}(\KW_\Q)_{0,0}(\pt) \tau^{st}_{m}
\]
given by $\tau_{2j+1}\mapsto \beta^{-n} \tau_{j-n}^{st}$. Lemma~\ref{lm:shift_by_beta} yields that 
\[
\tau^{st}_{m}=\beta_l^{-1}\cprod\tau^{st}_{m-1}\cprod \beta_r,
\]
whence $\tau^{st}_{m}=\beta_l^{-m}\cprod\tau_0^{st}\cprod \beta_r^{m}$ and
\[
(\KW_\Q)_{0,0}(\KW_\Q)= \bigoplus_{m\in \Z}(\KW_\Q)_{0,0}(\pt) \beta_l^{-m}\cprod \tau_{0}^{st}\cprod \beta_r^{m}.
\]

In order to check that $\tau_0^{st}=u_{\KW_\Q}\wedge u_{\KW_\Q}$ (whence $\beta_l^{-m}\cprod \tau_{0}^{st}\cprod \beta_r^{m}=\beta_l^{-m}\cprod \beta_r^{m}$) recall that $s^\vee_1=\chi_1$ and consider the following diagram.
\[
\xymatrix@C=4pc{
 & \KW_\Q\wedge (\Sigma^\infty_{\HP^1} \HP^1 \{-1\}) \ar[r]^{\id_{\KW_\Q}\wedge i}\ar[d]^\cong    &\KW_\Q\wedge (\Sigma^\infty_{\HP^1} \HGr \{-1\}) \ar[d]^{\cong} \\
\mathbb{S} \ar[r]^{u_{\KW_\Q}\wedge \id_{\mathbb{S}}} \ar[ru]^(0.4){u_{\KW_\Q}\wedge \Sigma^{-1}_{\HP^1} \chi_1}  \ar[rd]_{\tau_0^{st}} &  \KW_\Q \wedge \mathbb{S}  \ar[d]^{\id_{\KW_\Q}\wedge u_{\KO}} & \KW_\Q \wedge \Tr_1\KO \ar[dl]^{\id_{\KW_\Q}\wedge j}\\
& \KW_\Q\wedge \KO
}
\]
Here 
\begin{itemize}
\item
$i$ is induced by the canonical embedding $\HProj^1\to \HGr$,
\item
$j$ is the canonical morphism $\Tr_1\KO\to \KO$.
\end{itemize}
The right half of the diagram commutes by Lemma~\ref{lm:unit_KO}. The upper triangle commutes by Corollary~\ref{cor:normalization}, the outer contour commutes by the definition of $\tau_0^{st}$, thus the lower triangle commutes as well and the claim follows.
\end{proof}

\begin{remark}
One can restate Theorem~\ref{thm:cooperations} as follows.  Let 
\[
\mathcal{B}=\left( \Sigma^{8m,4m}\beta^m \right)_{m\in \Z}  \colon \bigoplus_{m\in \Z} \mathbb{S}\wedge \Sph^{8m,4m} \to \KW_\Q
\]
be the morphism given by $ \Sigma^{8m,4m} \beta^m\colon \mathbb{S}\wedge \Sph^{8m,4m} \to \KW_\Q$. Then the induced homomorphism in homology
\[
\mathcal{B}_{\KW_\Q}\colon (\KW_\Q)_{*,*}(\bigoplus_{m\in \Z}\mathbb{S}\wedge \Sph^{8m,4m})\to (\KW_\Q)_{*,*}(\KW_\Q)
\]
is an isomorphism.
\end{remark}

Now we turn to the description of integral cooperations.

\begin{theorem} \label{thm:cooperations_integral}
Let $\mathrm{M}$ be the abelian subgroup of $\Q[v,v^{-1}]$ generated by polynomials 
\[
f_{j,n}=\frac{v^{-n}\prod_{i=0}^{j-1}(v-(2i+1)^2)}{4^j(2j)!},
\]
$j\ge 0, n\in \Z$. Then there are canonical isomorphisms of left $\KW_{0,0}(\pt)\cong \W(k)$-modules
\[
\KW_{p,q}(\KW) \cong \left[ \begin{array}{ll} \W(k)\otimes_\Z \mathrm{M}, & 4 \mid p-q, \\ 0, & \text{otherwise}. \end{array} \right.
\]
Rationally $\W_\Q(k)\otimes_\Z \mathrm{M}\cong (\KW_\Q)_{r,r-4t}(\KW)$ is given by
\[
v^m\mapsto \eta^{r-8t}\beta_l^{t-m}\cprod \beta_r^{m}
\]
in the notation of Theorem~\ref{thm:cooperations}.
\end{theorem}
\begin{proof}
Applying the reasoning dual to the one used in the beginning of the proof of Theorem~\ref{thm:rational_stable_operations_KW} we obtain that
\[
\KW_{*,*}(\KW)=\varinjlim \bigoplus_{i\ge 1} \KW_{*+8n-4i+4,*+4n-2i+2}(\pt)\overline{b}^\vee_{i}.
\]
Here $\overline{b}^\vee_{i}$ belongs to the submodule of $\KW_{*,*}(\HGr)$ dual to the indecomposable quotient $\IQ^{*,*}(\HGr)$ and satisfies $\langle b_i, \overline{b}^\vee_{i}\rangle=1$, $\langle b_l, \overline{b}^\vee_{i}\rangle=0$ for $l\neq i$. The limit is computed along the morphisms $S^\vee$ dual to the corresponding morphisms $S$ between indecomposable quotients. Recall that $S$ is induced by a desuspension of an appropriate morphism $\HP^1\wedge \HP^1\wedge \HGr\to \HGr$.

It follows from Lemma~\ref{lm:formal_group_law} that $S(b_i)$ is an $\Z[\beta,\beta^{-1}]$-linear combination of products of Borel classes $b_j$-s (cf. Lemma~\ref{lm:main_computation}), thus there exists a linear map
\[
S_\Z \colon \prod_{i\ge 1} \Z[\beta,\beta^{-1}] \overline{b}_{i} \to \prod_{i\ge 1} \Z[\beta,\beta^{-1}] \overline{b}_{i}.
\]
inducing 
\[
S \colon \prod_{i\ge 1}\KW^{*+8n-4i+12,*+4n-2i+6}(\pt)\overline{b}_i\to \prod_{i\ge 1}\KW^{*+8n-4i+4,*+4n-2i+2}(\pt)\overline{b}_i.
\]
Moreover, $S_\Z$ gives rise to the dual map
\[
S^\vee_\Z\colon \bigoplus_{i\ge 1} \Z[\beta,\beta^{-1}] \overline{b}^\vee_{i} \to \bigoplus_{i\ge 1} \Z[\beta,\beta^{-1}] \overline{b}^\vee_{i}.
\]
and 
\[
S^\vee\colon \bigoplus_{i\ge 1}\KW_{*+8n-4i+4,*+4n-2i+2}(\pt)\overline{b}^\vee_i\to \bigoplus_{i\ge 1}\KW_{*+8n-4i+12,*+4n-2i+6}(\pt)\overline{b}^\vee_i
\]
is given by $S^\vee=\id_{\KW_{*,*}(\pt)} \otimes_{\Z[\beta,\beta^{-1}]} S^\vee_\Z$. 

The proof of Lemma~\ref{lm:main_computation} yields
\[
S_\Z(\overline{s}_{2j})=0,\quad S_\Z(\overline{s}_{2j+1})=\beta(2j+1)^2\overline{s}_{2j+1} - 8 j(2j+1) \overline{s}_{2j-1}.
\]
From the Newton identities we have $\langle \overline{s}_i, \overline{b}^\vee_i\rangle=(-1)^{i+1}i$ and $\langle \overline{s}_l, \overline{b}^\vee_i\rangle=0$ for $l\neq i$. Combining this with the above we obtain
\begin{gather*}
\langle \overline{s}_{2j}, S^\vee_\Z(\overline{b}^\vee_i) \rangle= \langle S_\Z(\overline{s}_{2j}), \overline{b}^\vee_i \rangle=0,\\
\langle \overline{s}_{2j+1}, S^\vee_\Z(\overline{b}^\vee_i) \rangle= \langle S_\Z(\overline{s}_{2j+1}), \overline{b}^\vee_i \rangle=
\left[\begin{array}{ll} \beta^{-1}(2j+1)^3, & i=2j+1,\\ -8j(2j-1)(2j+1), & i=2j-1,\\ 0, & \text{otherwise}. \end{array}\right.
\end{gather*}
Hence $S^\vee_\Z(\overline{b}^\vee_{2j})=0$ and $S^\vee_\Z(\overline{b}^\vee_{2j+1})=(2j+1)^2\beta\overline{b}^\vee_{2j+1}-8(j+1)(2j+1) \overline{b}^\vee_{2j+3}$. Thus we may drop all the $\overline{b}_{2j}^\vee$ obtaining
\[
\KW_{*,*}(\KW)=\varinjlim \bigoplus_{j\ge 0} \KW_{*+8(n-j),*+4(n-j)}(\pt)\overline{b}^\vee_{2j+1}.
\]
Specifying to the degree $(p,q)$, $4\nmid p-q$, we obtain $\KW_{p,q}(\KW)=0$ since 
\[
\KW_{p+8(n-j),q+4(n-j)}(\pt)\cong \W^{[q-p-4(n-j)]}(k)=0.
\]
In view of the periodicities given by cap-product with $\eta$ and $\beta$ from now on we deal with $\KW_{0,0}(\KW)$. 

We have
\[
\KW_{0,0}(\KW)=\varinjlim_n \bigoplus_{j\ge 0} \KW_{8(n-j),4(n-j)}(\pt)\overline{b}^\vee_{2j+1}=\varinjlim_n \bigoplus_{j\ge 0} \W(k)\beta^{n-j}\overline{b}^\vee_{2j+1}
\]
where the colimit is computed with respect to the morphism 
\[
S^\vee\colon \bigoplus_{j\ge 0} \W(k)\beta^{n-j}\overline{b}^\vee_{2j+1} \to \bigoplus_{j\ge 0} \W(k)\beta^{n+1-j}\overline{b}^\vee_{2j+1}
\]
given by 
\[
S^\vee(\beta^{n-j}\overline{b}^\vee_{2j+1})=(2j+1)^2\beta^{n+1-j}\overline{b}^\vee_{2j+1}-8(j+1)(2j+1)\beta^{n-j}\overline{b}^\vee_{2j+3}.
\]
Colimit commutes with tensor product, thus
\[
\KW_{0,0}(\KW)= \W(k) \otimes_\Z \left( \varinjlim_n \bigoplus_{j\ge 0} \Z\beta^{n-j} \overline{b}^\vee_{2j+1} \right)
\]
with the morphisms 
\[
S^\vee_\Z\colon \bigoplus_{j\ge 0} \Z\beta^{n-j} \overline{b}^\vee_{2j+1} \to \bigoplus_{j\ge 0} \Z\beta^{n+1-j} \overline{b}^\vee_{2j+1}
\]
in the bases $\{\beta^{n-j} \overline{b}^\vee_{2j+1}\}_{j\ge 0}$ and $\{\beta^{n+1-j}\overline{b}^\vee_{2j+1}\}_{j\ge 0}$ given by 
\[
\begin{pmatrix}
a_1 & 0 & 0 & 0 & \hdots \\
c'_3 & a_3 & 0 & 0 & \hdots \\
0 & c'_5 & a_5 & 0 & \hdots \\
0 & 0 & c'_7 & a_7 & \hdots \\
\vdots & \vdots & \vdots & \vdots & \ddots 
\end{pmatrix},
\]
where $a_{2j+1}=(2j+1)^2$ and $c'_{2j+1}=-8j(2j-1)$. 

The terms in the last colimit are torsion-free, so the canonical morphism
\[
\varinjlim_n \bigoplus_{j\ge 0} \Z\beta^{n-j}\overline{b}^\vee_{2j+1} \to \varinjlim_n \bigoplus_{j\ge 0} \Q\beta^{n-j}\overline{b}^\vee_{2j+1}
\]
is injective. One computes the right-hand side colimit as in the proof Theorem~\ref{thm:cooperations}. Denote 
\[
T^\vee_\Z=(\beta^{-1}\cap -)\circ S^\vee_\Z\colon \bigoplus_{j\ge 0} \Z\beta^{n-j}\overline{b}^\vee_{2j+1} \to \bigoplus_{j\ge 0} \Z\beta^{n-j} \overline{b}^\vee_{2j+1}
\]
and choose a basis of $\bigoplus_{j\ge 0} \Q\beta^{n-j}\overline{b}^\vee_{2j+1}$ to be
\[
\{\beta^n\overline{b}^\vee_1, T_\Q^\vee(\beta^n\overline{b}^\vee_1), (T_\Q^\vee)^2(\beta^n\overline{b}^\vee_1),\hdots \}.
\]
In these bases $S^\vee_\Q$ is a shift, thus
\[
\varinjlim_n \bigoplus_{j\ge 0} \Q\beta^{n-j}\overline{b}^\vee_{2j+1}= \bigoplus_{m\in \Z} \Q \cdot[\beta_l^{-m}\cprod\beta_r^m]
\]
with the canonical morphisms  
\[
\bigoplus_{j\ge 0} \Q\beta^{n-j} \overline{b}^\vee_{2j+1}\to \bigoplus_{m\in \Z} \Q\beta_l^{-m}\cprod\beta_r^m
\]
given by $(T_\Q^\vee)^m(\beta^{n}\overline{b}^\vee_{1})\mapsto \beta_l^{-m}\cprod\beta_r^m$ (the notation is consistent with the one used in the proof of Theorem~\ref{thm:cooperations}). The limit $\varinjlim\limits_n \bigoplus_{j\ge 0} \Z\beta^{n-j}\overline{b}^\vee_{2j+1}$ is the union of the images for the canonical morphisms
\[
\phi_n\colon \bigoplus_{j\ge 0} \Z\beta^{n-j}\overline{b}^\vee_{2j+1} \to \bigoplus_{m\in \Z} \Q \beta_l^{-m}\cprod\beta_r^m.
\]
We claim that these morphisms are given by 
\[
\phi_n(\beta^{n-j} \overline{b}^\vee_{2j+1})
=\frac{(\beta_l^{-m}\cprod\beta_r^m)\prod_{i=0}^{j-1}(\beta_l^{-1}\cprod\beta_r - a_{2i+1})}{\prod_{i=1}^{j}{c'_{2i+1}}},
\]
where the multiplication on the right-hand side is componentwise, i.e. 
\[
(\beta_l^{-n}\cprod\beta_r^n)(\beta_l^{-m}\cprod\beta_r^m)=\beta_l^{-n-m}\cprod\beta_r^{n+m}.
\]
Indeed, for $j=0$ we have $\phi_n(\beta^{n} \overline{b}^\vee_{1})=\beta_l^{n}\cprod\beta_r^{-n}$. The general case follows from the equalities
\begin{multline*}
\phi_{n+1}(a_{2j-1}\beta^{n+1-j} \overline{b}^\vee_{2j-1}+c'_{2j+1}\beta^{n-j} \overline{b}^\vee_{2j+1})=\\
=\phi_{n+1}(S^\vee_\Z(\beta^{n-j+1} \overline{b}^\vee_{2j-1}))=\phi_n(\beta^{n-j+1} \overline{b}^\vee_{2j-1}).
\end{multline*}
The claim of the theorem follows.
\end{proof}

\begin{remark}
It follows from the above theorem applied to $k=\mathbb{R}$ (or any other field satisfying $W(k)=\Z$) that $M$ is an algebra for the usual multiplication of polynomials, i.e. that products $f_{j_1,n_1}f_{j_2,n_2}$ can be expressed as linear combinations of $f_{j,n}$'s. For example we have
\[
f_{1,0}^2=9f_{1,-1}+198f_{2,-1}+720f_{3,-1}.
\]
\end{remark}


\begin{thebibliography}{XXXX}

\bibitem[Ad74]{Ad74}
J.~F.~Adams,
\emph{Stable homotopy and generalized homology,}
Univ. of Chicago Press, 1974

\bibitem[AHS71]{AHS71}
J.~F.~Adams, A.~S.~Harris, R.~M.~Switzer,
\emph{Hopf algebras of cooperations for real and complex K-theory},
Proc. London Math. Soc., vol. 23 (1971), pp. 385--408.

\bibitem[ALP15]{ALP15}
A.~Ananyevskiy, M.~Levine, I.~Panin,
\emph{Witt sheaves and the $\eta$-inverted sphere spectrum},
arXiv:1504.04860, to appear in J. Topology.

\bibitem[An16]{An12}
A.~Ananyevskiy,
\emph{On the relation of special linear algebraic cobordism to Witt groups,}
Homotopy, Homology and Applications, 18:1 (2016), 205--230

\bibitem[An15]{An15}
A.~Ananyevskiy,
\emph{The special linear version of the projective bundle theorem,}
Compositio Math., 151 (2015), no. 3, pp. 461--501

\bibitem[Bal99]{Bal99}
P.~Balmer,
\emph{Derived Witt groups of a scheme,}
J. Pure Appl. Algebra, 141 (1999), 101--129

\bibitem[Bal05]{Bal05}
P.~Balmer,
\emph{Witt groups,}
Handbook of K-theory. Vol. 1, 2, Springer, Berlin, 2005, pp. 539--576

\bibitem[Br84]{Br84}
G.W.~Brumfiel,
\emph{Witt rings and K-theory,}
Rocky Mountain J. of Math. 14, no. 4 (1984), pp. 733--765

\bibitem[CD09]{CD09}
D.-C.~Cisinski, F.~Deglise,
\emph{Triangulated categories of mixed motives},
 arXiv:0912.2110

\bibitem[GN03]{GN03}
S.~Gille, A.~Nenashev, 
\emph{Pairings in triangular Witt theory,}
J. Algebra, 261 (2003), 292--309

\bibitem[Hor05]{Hor05}
J.~Hornbostel,
\emph{$\A^1$-representability of Hermitian $K$-theory and Witt groups,}
Topology 44 (2005), no. 3, 661--687

\bibitem[Jar00]{Jar00}
J.~F.~Jardine,
\emph{Motivic symmetric spectra,}
Doc. Math., 5 (2000), 445--552

\bibitem[Kne77]{Kne77}
M.~Knebusch,
\emph{Symmetric bilinear forms over algebraic varieties,}
Publ. Math. Debrecen, 46 (1977), 103--283

\bibitem[KSW16]{KSW15}
M.~Karoubi, M.~Schlichting, C.~Weibel,
\emph{The Witt group of real algebraic varieties,}
J Topology 9:4 (2016), 1257--1302

\bibitem[MV99]{MV99}
F.~Morel, V.~Voevodsky,
\emph{$\A^1$-homotopy theory of schemes,}
Publ. Math. IHES, 90 (1999), 45--143

\bibitem[PPR09]{PPR09}
I.~Panin, K.~Pimenov, O.~R\"ondigs,
\emph{On Voevodsky’s algebraic K-theory spectrum,}
Algebraic topology, vol. 4 of Abel Symp., Springer, Berlin, 2009, pp. 279--330.

\bibitem[PW10a]{PW10a}
I.~Panin and C.~Walter,
\emph{Quaternionic Grassmannians and Pontryagin classes in algebraic geometry,}
arXiv:1011.0649.

\bibitem[PW10b]{PW10b}
I.~Panin and C.~Walter,
\emph{On the algebraic cobordism spectra MSL and MSp,}
arXiv:1011.0651.


\bibitem[PW10c]{PW10c}
I.~Panin and C.~Walter,
\emph{On the motivic commutative spectrum BO,}
arXiv:1011.0650.


\bibitem[Sch10a]{Sch10a}
M.~Schlichting,
\emph{Hermitian K-theory of exact categories,}
J. K-theory 5 (2010), no. 1, 105 -- 165

\bibitem[Sch10b]{Sch10b}
M.~Schlichting,
\emph{The Mayer-Vietoris principle for Grothendieck-Witt groups of schemes,}
Invent. Math., 179 (2010), no. 2, 349 -- 433

\bibitem[Sch17]{Sch12}
M.~Schlichting,
\emph{Hermitian K-theory, derived equivalences and Karoubi's fundamental theorem,}
to appear in J. Pure Appl. Algebra, http://dx.doi.org/10.1016/j.jpaa.2016.12.026

\bibitem[ST15]{ST13}
M.~Schlichting, G.~Tripathi,
\emph{Geometric models for higher Grothendieck-Witt groups in $\A^1$-homotopy theory,}
Math. Annalen, 362:3 (2015), 1143--1167

\bibitem[V98]{V98}
V.~Voevodsky,
\emph{$\A^1$-homotopy theory,}
Doc. Math., Extra Vol. I (1998), pp. 579 -- 604

\bibitem[Wal03]{Wal03}
C.~Walter, 
\emph{Grothendieck-Witt groups of triangulated categories,}
K-theory preprint archive, 2003.


\end{thebibliography}
\end{document}